\documentclass[11pt,a4paper,reqno]{amsart}
\usepackage{amssymb,amsmath,amsthm}
\usepackage{epsfig,esint}

\makeatletter
\ifx\SetFigFont\undefined\gdef\SetFigFont#1#2#3#4#5{\reset@font\fontsize{#1}{#2pt}\fontfamily{#3}\fontseries{#4}\fontshape{#5}\selectfont}\fi

\font\script=rsfs10 at 11pt
\def\H{{\mbox{\script H}\;}}
\def\N{\mathbb{N}}
\def\R{\mathbb{R}}
\def\S{\mathbb{S}}
\def\PE{P_{\rm eucl}}
\def\VE#1{|#1|_{\rm eucl}}
\def\eps{\varepsilon}
\def\IF{{\mathfrak I}}
\def\XXint#1#2#3{{\setbox0=\hbox{$#1{#2#3}{\int}$} \vcenter{\vspace{-1pt}\hbox{$#2#3$}}\kern-.5\wd0}}
\def\Xint#1{\mathchoice {\XXint\displaystyle\textstyle{#1}}{\XXint\textstyle\scriptstyle{#1}}{\XXint\scriptstyle\scriptscriptstyle{#1}}{\XXint\scriptscriptstyle\scriptscriptstyle{#1}}\!\int}
\def\intmed{\Xint{-}}
\def\step#1#2{\par\noindent{\underline{\it Step~#1.}}\emph{ #2}\\}
\def\Chi#1{\hbox{{\large $\chi$}{\Large $_{_{#1}}$}}}

\numberwithin{equation}{section}

\newtheorem{theorem}{Theorem}[section]
\newtheorem{lemma}[theorem]{Lemma}
\newtheorem{conj}[theorem]{Conjecture}
\newtheorem{prop}[theorem]{Proposition}
\newtheorem{corollary}[theorem]{Corollary}
\newtheorem{definition}[theorem]{Definition}
\newtheorem{remark}[theorem]{Remark}
\newtheorem{example}[theorem]{Example}
\newtheorem{op}[theorem]{Open Problem}

\title[Isoperimetric problem with density]{Existence of isoperimetric regions in $\R^n$ with density}
\author{Frank Morgan}
\address{Department of Mathematics and Statistics, Williams College, Williamstown, MA 01267}
\email{Frank.Morgan@williams.edu}
\author{Aldo Pratelli}
\address{Dipartimento di Matematica, Universit\`a di Pavia ``F. Casorati'', via Ferrata 1, 27100 Pavia, Italy}
\email{aldo.pratelli@unipv.it}

\begin{document}

\begin{abstract}
We prove the existence of isoperimetric regions in $\R^n$ with density under various hypotheses on the growth of the density. Along the way we prove results on the boundedness of isoperimetric regions.
\end{abstract}

\maketitle

\tableofcontents

\section{Introduction}

There has been a recent surge of interest in Riemannian manifolds with a positive ``density'' function that weights volume and area (see~\cite{newM1,newM2}) and in particular in the isoperimetric problem of minimizing weighted perimeter for given weighted volume. Whether isoperimetric regions exist in $\R^n$ with density depends on the density. We present the following:
\par\medskip\noindent
{\bf Conjecture~\ref{conj}.} {\it Let $f$ be a radial, increasing density on $\R^n$. Then isoperimetric sets exist for all volumes.}

Proposition~\ref{ex:nonex} gives a non-radial increasing density for which existence fails. 

Following more restrictive results of Rosales et al.~\cite{BCMR}, our Theorem~\ref{thm2.2} proves the conjecture if the density approaches infinity, even if not increasing. Propositions~\ref{non-existence-d-increasing} and~\ref{Prop2.1A} provide examples to show that neither hypothesis can be simply deleted.

If the density approaches a finite limit $a$ at infinity, we do not need to assume $f$ radial, but we need some assumption on the growth (Theorems~\ref{Thm2A}, \ref{6.11}, \ref{Thm2B}) and some condition to make isoperimetric regions bounded (Corollary~\ref{corbdd}). Our results cover all the standard examples (Remark~\ref{Rem7.16}).

The growth hypotheses are of two types. Theorems~\ref{Thm2A} and~\ref{6.11} assume that the density approaches
the limiting value slowly in some sense. Theorem~\ref{Thm2B} assumes an averaging condition on the density, weaker than superharmonicity (Corollary~\ref{Thm2D}).

We prove boundedness for increasing densities in three cases. Proposition~\ref{propisobound2} handles $\R^2$ without assuming the density radial. Proposition~\ref{nonbdd} shows by example that in general dimensions, further hypotheses are necessary. Proposition~\ref{thba} handles $\R^n$ with radial density. Proposition~\ref{lastfrank} instead assumes the density ${\rm C}^1$ and Lipschitz.

{\bf The proofs.} The main step of the existence proofs (Proposition~\ref{generalcase}) is to show that there are balls arbitrarily far from the origin with ``mean density'' at most $a$. Given that, the proof of the existence of an isoperimetric region of prescribed volume proceeds as follows. Take a minimizing sequence converging to a limit $F$. The problem is that some volume may be lost to infinity, with mean density $a$. Since $F$ is bounded, the missing volume may be replaced by a ball far from the origin of mean density $a$. The hardest part is to find the right growth conditions to provide the distant balls with mean density at most $a$.

Section~\ref{sect5} discusses the convexity of isoperimetric sets. Section~\ref{seccon} gives our main existence results, and Section~\ref{sectopenpbs} collects some open problems.

This paper focuses on $\R^n$ for $n\geq 2$. In $\R^1$, most of our questions are trivial and much finer results are already known (see e.g.~\cite{BCMR}).

\section{Preliminaries\label{succor}}

Let us first set some notation and list some known results. We confine attention to Euclidean space $\R^n$. For a background on geometric measure theory, see Giusti~\cite{G} or Morgan~\cite{M1}. The ball and the sphere of radius $r$ are denoted respectively by
\begin{align*}
B(r) := \big\{ x\in \R^n:\, |x| \leq r\big\} \,, && S(r) := \big\{ x\in \R^n:\, |x| = r\big\} = \partial B(r) \,.
\end{align*}
The letter $f$ will always denote the density (lower-semicontinuous positive function on $\R^n$) that we use to calculate perimeters and volumes. Hence, given any set $E$ of locally finite perimeter, we will denote its volume and perimeter by
\begin{align*}
\big| E \big|_f := \int_{E} f(x) \, dx\,, && P_f (E) := \int_{\partial E} f(x) \, d\H^{n-1}(x)\,.
\end{align*}
By $\partial E$ we denote the essential boundary of $E$, which coincides with the usual boundary of $E$ if it is a smooth or piecewise affine set. For a given positive volume $V>0$, we set
\[
\IF_f(V) := \inf \Big\{ P_f(E):\, |E|_f=V \Big\}\,.
\]
The function $\IF_f$ is usually referred to as the \emph{isoperimetric function} or \emph{isoperimetric profile}, while an \emph{isoperimetric set} is any set $E$ such that $P_f(E)= \IF_f\big(|E|_f\big)$. We will avoid the subscript $f$ when there is no risk of confusion. The following regularity result is known; see for instance~\cite[Proposition~3.5, Corollary~3.8]{M2}.
\begin{theorem}\label{thm1.0}
Let $f$ be a smooth or $C^{k-1,\alpha}$ density on $\R^n$. Then the boundary of an isoperimetric set is a smooth or $C^{k,\alpha}$ submanifold except on a singular set of Hausdorff dimension at most $n-8$.
\end{theorem}

Given a set of finite perimeter $E$, for $\H^{n-1}$-a.e. $x\in\partial E$ the \emph{outer} normal $\nu_E(x)$ is well defined. Sometimes, when there is no risk of confusion, we will simply write $\nu(x)$. For $x\in\partial E$, let $H_0(x,E)$ denote the inward Euclidean mean curvature, defined if $E$ is twice differentiable at $x$. For convenience, we take the mean curvature to be the sum rather than the average of the principal curvatures, so that it is $n-1$ rather than $1$ for the unit sphere in $\R^n$. If $E\subseteq \R^2$ is locally the region above the graph of a function $\tau$, twice differentiable at $x$, then the (upward) mean curvature is given by
\[
H_0\big( (x,\tau(x)\big) = \frac{\tau''(x)}{\big( 1+\tau'(x)^2\big)^{3/2}}\,.
\]
(see~\cite[p. 6]{M3}). A twice differentiable connected planar set is convex if and only if its boundary has nonnegative (inward) mean curvature. A twice differentiable connected subset of $\R^n$ is called \emph{mean-convex} if it has nonnegative mean curvature.\par\bigskip

We now present a classical first variation formula and the notion of curvature in the case of Euclidean space $\R^n$ with a density $f$ (for more details, see \cite[Sect. 3]{BCMR}). To start, we define the function $v:\R^2\to\R$ such that the density $f$ can be expressed as
\[
f(x) = e^{v(x)}\,.
\]
A careful computation yields the following first order expansion formulae for perimeter and volume.
\begin{lemma}[{\bf First variation formulae}\hspace{0pt}]\label{lemma-expansion}
Let $E$ be a $C^2$ subset of $\R^n$. For any $C^2$ function $u: \partial E \to \R$ and a small positive number $\eps$, consider the set $E_\eps$ such that
\[
\partial E_\eps = \Big\{ x + \eps u(x) \nu(x):\, x\in \partial E \Big\}\,.
\]
Then the following first-order expansions for volume and perimeter of $E_\eps$ hold,
\begin{gather}
\big| E_\eps \big|_f = \big| E\big|_f +  \eps \int_{\partial E} u(x) f(x) \, d\H^{n-1}(x) + o(\eps)\,,\label{expansion-volume}\\
P_f\big( E_\eps \big) = P_f\big( E\big) +  \eps \int_{\partial E} \bigg( H_0 (x,E) + \frac{\partial v}{\partial \nu_E(x)} (x)\bigg) u(x) f(x) \, d\H^{n-1}(x) + o(\eps)\,.\label{expansion-wrong}
\end{gather}
\end{lemma}
In view of the above expansions, it is natural to give the following definition.
\begin{definition}\label{X.x}
Let $E$ be a set of finite perimeter. For any $x\in\partial E$ such that $\nu_E(x)$ exists (hence, for $\H^{n-1}-$a.e. $x\in\partial E$), we define the \emph{generalized curvature} with respect to the density $f=e^v$ as
\begin{equation}\label{X.x'}
H_f(x,E) := H_0 (x,E) + \frac{\partial v}{\partial \nu_E(x)} (x)\,.
\end{equation}
Again, when there is no risk of confusion, we will simply write $H_f(x)$, or even $H(x)$.
\end{definition}
As a consequence of this definition, (\ref{expansion-wrong}) simply reads as
\begin{equation}\label{expansion-perimeter}
P_f\big( E_\eps \big) = P_f\big( E\big) +  \eps \int_{\partial E} H_f (x,E) u(x) f(x) \, d\H^{n-1}(x) + o(\eps)\,,
\end{equation}
thus by~(\ref{expansion-volume}) (the boundary of) an isoperimetric set has always constant generalized curvature.\par\medskip
To conclude this introduction, we recall the following results.
\begin{theorem}[{\bf \cite{BCMR}, Theorem~3.10}\hspace{0pt}]
Consider a radial density $f=e^v$ on $\R^n$. If $v$ is convex, then balls about the origin are stable, while if $v$ is strictly
concave, then balls are unstable. More precisely, the second variation of perimeter for fixed volume for the ball $B(r)$ has the same sign as $v"(r)$.
\end{theorem}
\begin{figure}[htbp]
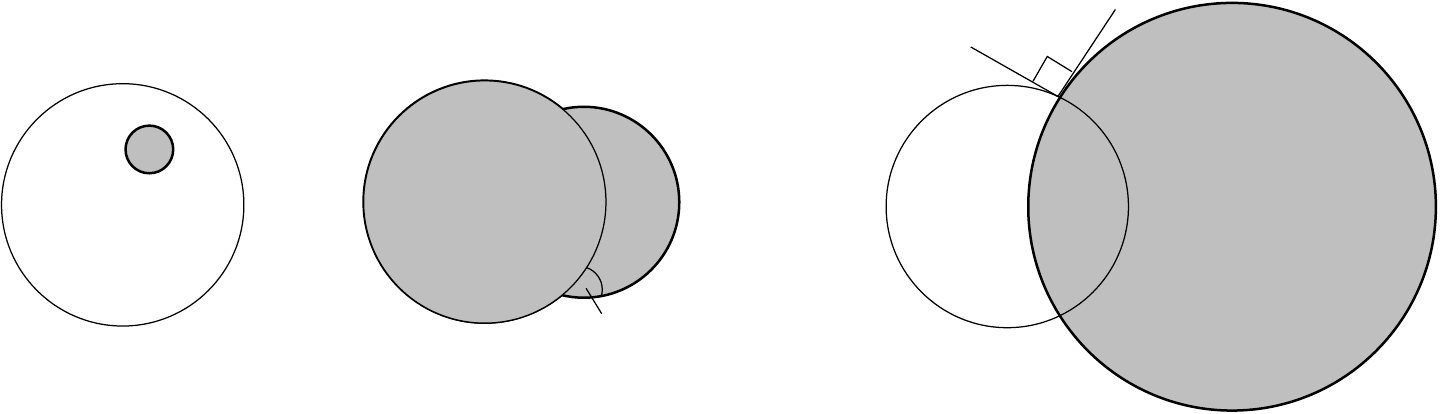\vspace{50pt}
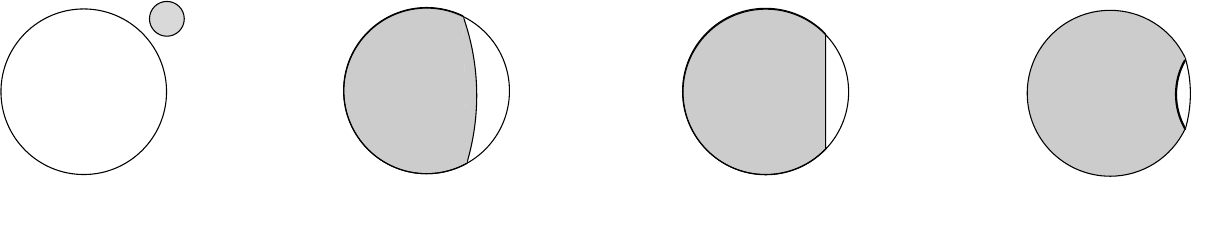
\caption{The first row shows isoperimetric sets for density $\lambda< 1$  inside the ball and density $1$ outside the ball; the second row for density $1$ inside the ball and density $\lambda < 1$ outside the ball (Example~\ref{ex11}). Some of these sets are neither symmetric nor convex. Figures courtesy Ca\~nete, Miranda, and Vittone~\cite[Figures 17 and 13]{CMV}; all rights reserved.}\label{figdav}
\end{figure}
\begin{example}[{\bf \cite{CMV}, Theorems~3.23, 3.20}\hspace{0pt}]\label{ex11}
Consider $\R^2$ with density $0<\lambda<1$ inside the unit ball and density $1$ outside the ball. This density is radial and nondecreasing (although neither log-concave or log-convex), but it admits some isoperimetric sets which are not convex or which do not contain the origin, as in the first line of Figure~\ref{figdav}. In particular, for prescribed volume slightly greater than $\lambda\pi$, the isoperimetric profile satisfies
\[
\IF(\lambda \pi+ \eps) \approx 2\lambda \pi + c_\lambda \sqrt{\eps}\, .
\]
Alternatively, consider $\R^2$ with density $1$ inside the open unit ball and density $0 < \lambda < 1$ outside the ball. This density is radial and nonincreasing, but again admits some isoperimetric sets which are not convex or which do not contain the origin, as in the second line of Figure~\ref{figdav}. For prescribed volume slightly less than $\pi$, the isoperimetric profile satisfies
\[
\IF(\pi-\eps) \approx 2 \lambda \pi + c_\lambda \sqrt{\eps}\,;
\]
in particular, it is sometimes decreasing.
\end{example}

\section{Existence of isoperimetric sets\label{sect2}}


Notice that, due to the different scaling of perimeter and volume, in regions with constant density the perimeter of a ball of given volume is larger when the density is larger. Hence, roughly speaking, isoperimetric sets try to go where the density is low. As a consequence, we can expect that a sequence of sets minimizing the perimeter stays close to the origin if the density diverges, while they wander far from the origin if the density tends to $0$. Therefore, for a density which decreases to $0$, isoperimetric sets often fail to exist. More generally, for a
density which decreases to any limit, isoperimetric sets often fail to exist. On the other hand, one naively expects that existence should hold when the density is increasing, and in particular if the density diverges to infinity. Unfortunately, this is not true for some bumpy densities (see~\cite[Example~2.6]{BCMR}).

\begin{prop}\label{non-existence-d-increasing}
There exists a smooth density $f$ on $\R^n$ approaching infinity at infinity for which no isoperimetric set exists. Indeed, the infimum perimeter to enclose every volume $V > 0$ is $\IF(V) = 0$.
\end{prop}
\begin{proof}
Let us start with any radial and increasing smooth density $f_0$ which is diverging to infinity. Then, take countably many disjoint balls $B_i$ moving away from the origin and such that $P_{f_0}(B_i) = 1/i^2$. Observe that the radii of these balls decrease to $0$ and the volumes $\big| B_i \big|_{f_0}$ of these balls converge to $0$ faster than $1/i^2$. For any ball $B_i$, with $i$ big enough, it is possible to \emph{increase} the density in the interior of the ball, still remaining smooth, in such a way that the volume of the ball increases up to $1/i$. Applying this procedure for all the balls, we find a smooth diverging density $f\geq f_0$ such that for each $i\in\N$ big enough one has
\begin{align*}
P_f(B_i) = \frac 1{i^2}\,, && \big| B_i \big|_f = \frac 1 i\,.
\end{align*}
Countable unions of such balls provide sets of arbitrary volume and arbitrarily small perimeter.
\end{proof}

\begin{prop}\label{Prop2.1A}
There exists a smooth radial density $f$ on $\R^n$ for which no isoperimetric set exists.
\end{prop}
\begin{proof}
It is sufficient to consider a smooth radial density $f$ such that $f(r)=1/r$ for $r$ large enough, which is then decreasing to $0$ at infinity. Consider a ball $B_{R,D}$ centered at distance $D\gg 1$ from the origin, and with radius $1\ll R\ll D$. The volume and the perimeter of this ball are approximatively
\begin{align*}
\big| B_{R,D}\big|_f \approx \frac{\omega_n R^n}{D} \,, && P_f\big(B_{R,D}\big) = \frac{n\omega_n R^{n-1}}{D}\,.
\end{align*}
Therefore, for any given volume $V>0$ and for any distance $D$ big enough, it is possible to take some $R\approx (DV)^{1/n}$ such that the ball $B_{R,D}$ has exactly volume $V$. This ball has perimeter $P_f\big(B_{R,D}\big)\approx V^{\frac{n-1}n}/D^{\frac 1n}$, which is then arbitrarily small up to take $D$ big enough. As a consequence, $\IF(V)=0$ for any $V>0$, that is, no isoperimetric sets exist.
\end{proof}

D\'\i az et al.~\cite[Prop. 7.3]{D} show that no isoperimetric set exists in $\R^n$ with density $r^{-p}$ for $0 < p\leq n$. Our Proposition~\ref{ex:nonex} will provide a non-decreasing density on $\R^n$ for which there is no isoperimetric set of unit volume.

We can now show a positive result, stating that isoperimetric sets exist for all volumes if the density is radial and diverging. This result was known before only in $\R^2$~\cite[Theorem~2.5]{BCMR}, or under additional hypotheses in $\R^n$~\cite[Theorem~2.1]{BCMR}. The counterexamples of Propositions~\ref{non-existence-d-increasing} and~\ref{Prop2.1A} ensure that both assumptions are necessary.

\begin{theorem}\label{thm2.2}
Assume that $f$ is a (lower-semicontinuous) radial density on $\R^n$ which diverges to infinity. Then there exist isoperimetric sets for all volumes.
\end{theorem}
\begin{proof}
The basic idea of the proof is that if some volume goes off to infinity, there must be lots of tangential or radial perimeter.\par
Fix a volume $V>0$. By approximation, there is a sequence of smooth sets $E_j$ with volume $|E_j|=V$ and with $P(E_j)\searrow \IF(V)$. By the standard compactness results for sets (see for instance~\cite{AFP,M1}), we can extract a subsequence converging to a limit set $E$ with $|E|=V$ as soon as
\begin{equation}\label{upthere}
\lim_{R\to \infty} \limsup_{j\to \infty} \big| E_j \setminus B(R)\big| = 0\,.
\end{equation}
Since by lower semicontinuity one has $P(E)\leq \liminf P(E_j)=\IF(V)$, this means that $\IF(V)$ is in fact a minimum whenever condition~(\ref{upthere}) holds. As a consequence, if the result were not true, there would be then some $\eps>0$ such that, for each $R>0$ (and up to a subsequence)
\begin{equation}\label{non-vanishing}
\big| E_j \setminus B(R)\big| \geq \eps
\end{equation}
for all $j$ large enough (depending on $R$, of course!). \par
Let us then fix a sufficiently big number $R$, to be specified later, and fix also an index $j$ for which~(\ref{non-vanishing}) holds. Call $S_j(r)$ the area of the slice of $E_j$ at distance $r$ from the origin, that is,
\[
S_j(r) := \H^{n-1} \big( E_j \cap S(r) \big)\,.
\]
Hence, (\ref{non-vanishing}) reads as
\begin{equation}\label{refined-nv}
\int_R^{+\infty} f(r) S_j(r) \,dr \geq \eps\,.
\end{equation}
If we call
\begin{align*}
M_j = \max \{ S_j(r):\, r\geq R\}\,, && f_- = \min\{ f(r):\, r\geq R \}\,,
\end{align*}
then of course
\begin{equation}\label{Mjsmall}
P\big(E_j\big) \geq M_j f_-\,.
\end{equation}
Since we can choose $R$ in such a way that $f_-$ is arbitrarily big, it is admissible to assume that $M_j$ is small. In particular, for each $r\geq R$ the slice $E_j\cap S(r)$ is a small portion of the sphere $S(r)$. Hence, denoting by $p_j(r)$ the relative perimeter of $E_j$ in the slice $E_j\cap S(r)$, i.e.,
\[
p_j(r) := \H^{n-2} \Big( \partial \big( E_j \cap S(r)\big)\Big)\,,
\]
the standard isoperimetric inequality on the sphere $\S^{n-1}$ says that
\[
p_j(r) \geq c_n S_j(r)^{\frac{n-2}{n-1}} \qquad \forall\, r\geq R\,,
\]
where $c_n$ is a suitable dimensional constant. By Vol'pert Theorem (see~\cite{V}, or~\cite[Theorem~3.108]{AFP}), we know that the equality
\[
\partial \big( E_j \cap S(r)\big) = \big( \partial  E_j \big) \cap S(r)
\]
holds $\H^{n-2}$-a.e. for almost all $r$; hence by co-area formula (see~\cite{AFP, M1}) one has
\[
P(E_j) \geq \int_R^\infty p_j(r) f(r)\,dr
\geq c_n \int_R^\infty S_j(r)^{\frac{n-2}{n-1}} f(r)\,dr
\geq \frac{c_n}{M_j^{\frac{1}{n-1}}} \,\int_R^\infty S_j(r) f(r)\,dr
\geq \frac{c_n}{M_j^{\frac{1}{n-1}}} \, \eps\,,
\]
thanks to~(\ref{refined-nv}). Using now~(\ref{Mjsmall}), we derive
\[
P(E_j)^{\frac{n}{n-1}} \geq c_n\, f_-^{\frac{1}{n-1}} \eps\,,
\]
which gives a contradiction with the optimality of the sequence $E_j$, since $\eps>0$ is given while $f_-$ can be taken arbitrarily big if $R\gg 1$.
\end{proof}


\section{On the monotonicity of $\IF$}

Let us consider now another question which seems quite reasonable, that is, is it true that the isoperimetric profile $\IF$ is increasing? In other words, is it true that to enclose more volume one needs more perimeter? Theorem~\ref{Lambda-increasing} will give an affirmative answer for increasing densities. On the other hand, the answer is negative in the case of a finite total measure $\overline V=\int f<\infty$, for which $\IF(\overline V-V) = \IF(V)$ because a set and its complement have the same perimeter. 

In view of this observation, it may seem reasonable to guess that the isoperimetric function $\IF$ is increasing if the total volume of $\R^n$ is infinite, or at least if the density $f$ is diverging at infinity. Such a guess is wrong.

\begin{prop}\label{steepest}
There exists a smooth, diverging density $f$ on $\R^n$ such that $\IF(1)=0$ and $\IF(t)\geq |1-t|^{(n-1)/n}$ for all $1/2 \leq t \leq 3/2$.
\end{prop}
\begin{proof}
The result will be achieved with a modification of the argument of Proposition~\ref{non-existence-d-increasing}. We start with a smooth, radial, increasing and diverging density $f_0\geq 2$ for which
\begin{equation}\label{startingcase}
\IF_{f_0} (t) \geq 4 t^{(n-1)/n} \qquad \forall\, 0\leq t \leq 2\,.
\end{equation}
(Notice that this is clearly possible, since for the standard Euclidean density $f_E\equiv 1$ one has $\IF_{f_E}(t)=C(n)t^{(n-1)/n}$, and then it is enough to select $f_0$ big and diverging slowly enough). Let us also select a sequence of open balls $B_i={\rm Int} \big(B(x_i,r_i)\big)$ far from each other and from the origin, and with $P_{f_0}(B_i) = 1/i$. Notice that, since $f_0$ is diverging, each ball $B_i$ can be chosen to have a volume much smaller than $1/i^2$, and a radius $r_i$ much smaller than $1/i$. Call $r_i^+=\sqrt{r_i}$, so that $r_i\ll r_i^+\ll 1$, and let us also call $A_i={\rm Int} \big(B(x_i,r_i^+)\big)\setminus B(x_i,r_i)$ the open annulus of radii $r_i$ and $r_i^+$, and $\widetilde B_i=B(x_i,r_i^+)$ the ball with the same center as $B_i$ and radius $r_i^+$.\par
Let now $K_i$ be a big constant such that
\[
\int_{B_i} f_0 + K_i = 1\,,
\]
and define the density
\[
\tilde f(x) = \left\{\begin{array}{ll}
f_0(x)+K_i &\hbox{if $r_i\neq |x-x_i|<r_i^+$}\,,\\
f_0(x) &\hbox{otherwise}
\end{array}\right.
\]
(notice that $\tilde f$ is well-defined since the balls $\widetilde B_i$ are disjoint). By construction, $|B_i|_{\tilde f}=1$ for every $i$, while $P_{\tilde f}(B_i)=P_{f_0}(B_i)=1/i$, thus $\IF_{\tilde f}(1)=0$. We aim to show that, for a generic set $E\subseteq \R^n$ with $|E|_{\tilde f} \in (0,2)$, one has
\begin{equation}\label{benclaim}
P_{\tilde f}(E) \geq 2 \min\Big\{t^{(n-1)/n},\, |1-t|^{(n-1)/n}\Big\}\,,
\end{equation}
since the thesis will then directly follow just by substituting $\tilde f$ with a very similar but smooth density $f$ satisfying $f=\tilde f$ on every $\partial B_i$.\par
We take then a set $E\subseteq \R^n$ with $|E|_{\tilde f}\in (0,2)$, and we want to prove~(\ref{benclaim}). Let us define
\begin{align*}
E_1^i:= E \cap B_i\,, && E_2^i:= E\cap A_i\,, && E_3:= E \setminus \cup_i \big( E_1^i\cup E_2^i\big)\,.
\end{align*}
Since the balls $\widetilde B_i$ are far enough from each other, one has
\[
\H^{n-1}\Big(\partial E \setminus \cup_i \widetilde B_i \Big) \geq \frac{1}{3}\, \H^{n-1}\big( \partial E_3\big)\,;
\]
in particular, since $f_0$ is diverging slowly enough and $\tilde f = f_0$ out of $\cup_i \widetilde B_i$, by~(\ref{startingcase}) we have
\begin{equation}\label{sabamo}
\int_{\partial E \setminus \cup_i \widetilde B_i } \tilde f(x)\, d\H^{n-1}(x) \geq \frac{1}{4}\, P_{\tilde f} (E_3) = \frac{1}{4}\, P_{f_0} (E_3) \geq | E_3 |_{\tilde f}^{(n-1)/n}\,.
\end{equation}
Consider now $E_2^i$ for a generic $i$: since $r_i^+\gg r_i$ and since $|E_2|_{\tilde f}\leq 2\ll |A_i|_{\tilde f}$, one has 
\[
\H^{n-1}\Big(\partial E \cap  A_i \Big) \geq \frac{1}{3} \H^{n-1}\big( \partial E_2^i\big)\,.
\]
Hence, exactly as before we deduce
\begin{equation}\label{sabamo2}
\int_{\partial E \cap A_i } \tilde f(x)\, d\H^{n-1}(x) \geq \frac{1}{4}\, P_{f_0+K_i} (E_2^i) \geq |E_2^i|_{f_0 + K_i}^{(n-1)/n}
= | E_2^i |_{\tilde f}^{(n-1)/n}\,,
\end{equation}
where the second inequality comes from the fact that $f_0$ is almost constant in $A_i$ (say, $f_0\approx C_i$ in $A_i$) and then by~(\ref{startingcase}) we have
\[\begin{split}
P_{f_0+K_i} (E_2^i) &\approx \frac{C_i + K_i}{C_i}\,  P_{f_0}(E_2^i)
\geq \frac{C_i + K_i}{C_i}\,  4 \Big(|E_2^i|_{f_0}\Big)^{(n-1)/n}\\
&\approx \frac{C_i + K_i}{C_i}\,  4 \bigg(\frac{C_i}{C_i+K_i}\,|E_2^i|_{f_0+K_i}\bigg)^{(n-1)/n}
= \bigg(\frac{C_i + K_i}{C_i}\bigg)^{1/n}\,  4 \bigg(|E_2^i|_{f_0 + K_i}\bigg)^{(n-1)/n} \,.
\end{split}\]
Finally, let us consider a generic $E_1^i$. Recalling that $|B_i|_{\tilde f}=1$, one has
\[
\H^{n-1}\Big( \partial E \cap B_i \Big)  \geq \frac{1}{3}\, \min \bigg\{ \H^{n-1}\big( \partial E_3^i\big) ,\, \H^{n-1}\Big( \partial \big(B_i\setminus E_3^i\big)\Big) \bigg\}\,;
\]
hence, arguing exactly as in~(\ref{sabamo}) and~(\ref{sabamo2}), we obtain
\begin{equation}\label{sabamo3}
\int_{\partial E \cap B_i } \tilde f(x)\, d\H^{n-1}(x) \geq  \min \Big\{ \big(| E_3^i |_{\tilde f}\big)^{(n-1)/n}, \, \big(1 - | E_3^i |_{\tilde f}\big)^{(n-1)/n}\Big\}\,.
\end{equation}
This concludes the proof, because~(\ref{benclaim}) follows by~(\ref{sabamo}), (\ref{sabamo2}) and~(\ref{sabamo3}) since clearly
\[
P_{\tilde f} (E) = \int_{\partial E} \tilde f\, d\H^{n-1}
\geq \int_{\partial E \setminus \cup_i \widetilde B_i} \tilde f\, d\H^{n-1}+
\sum_i \bigg(\int_{\partial E \cap A_i } \tilde f\, d\H^{n-1}+
\int_{\partial E \cap B_i } \tilde f\, d\H^{n-1}\bigg)\,.
\]
\end{proof}

We conclude this section with our positive result, which says that $\IF$ is increasing if the density $f$ is non-decreasing in the sense of the following definition.
\begin{definition}\label{defnondecr}
Given a density $f$, not necessarily radial, we say that \emph{$f$ is non-decreasing} if for any $\theta\in \S^{n-1}$ the function $t\mapsto f(t\theta)$ is non-decreasing in $\R^+$.
\end{definition}

\begin{theorem}\label{Lambda-increasing}
Let $f$ be a density on $\R^n$ which is non-decreasing (but not necessarily radial). Then the isoperimetric profile $\IF$ is non-decreasing. Moreover, if isoperimetric sets exist for all volumes, then $\IF$ is strictly increasing. Indeed, if there exist an isoperimetric set of volume $V$, then $\IF(V')<\IF(V)$ for all $V'<V$.
\end{theorem}
\begin{proof}
Take any set $E$ of finite perimeter $P(E)$ and of volume $|E|=V$. For any $r>0$ such that $E \not \subseteq B(r)$, define
\[
E(r) := E \cap B(r) \subsetneq E\,.
\]
The main observation of the proof is the validity of
\begin{equation}\label{perimeter-decreases}
P\big( E(r)\big) < P (E)\,.
\end{equation}
To show this inequality, consider the projection $\alpha:\partial E \setminus B(r) \to S(r)$. It is immediate that $\alpha$ is strictly $1-$Lipschitz, and moreover the image $I(\alpha)$ of $\alpha$ satisfies
\begin{equation}\label{evennotbounded}
I(\alpha) \supseteq \Big( \partial E(r) \setminus \partial E \Big) \,,
\end{equation}
(which in turn is contained in $S(r)$). The validity of~(\ref{evennotbounded}) is trivial if $E$ is bounded, but it is true even if $E$ is unbounded. Indeed, let
\[
H=\Big(\partial E(r) \setminus \partial E \Big)\setminus I(\alpha)\,,
\]
and notice that $E$ containes the whole cone
\[
C = \Big\{ \lambda x:\, \lambda\geq 1,\, x\in H\Big\}\,.
\]
Since the density $f$ is increasing, the cone $C$ has infinite volume, and so does $E$, unless $\H^{n-1}(H)=0$, which then ensures us that the inclusion~(\ref{evennotbounded}) is true up to measure $0$.\par
As a consequence, by co-area formula~\cite{AFP, M1} and the fact that $f$ is increasing, one has
\[\begin{split}
P\big(E(r)\big) &= \int_{\partial E \cap B(r)} f(|x|)\,d\H^{n-1}(x) + \int_{\partial E(r) \setminus \partial E} f(\theta) \,d\H^{n-1}(\theta) \\
&< \int_{\partial E \cap B(r)} f(|x|)\,d\H^{n-1}(x) + \int_{\partial E \setminus B(r)} f\big(\alpha(x)\big)\,d\H^{n-1}(x) \\
&\leq \int_{\partial E} f(|x|)\,d\H^{n-1}(x) = P(E)\,,
\end{split}\]
so that~(\ref{perimeter-decreases}) has been established. Now, for every $0<V'<V$ one has the existence of some $r(V')$ with the property that
\[
\big| E(r(V'))\big| = V'\,,
\]
hence
\begin{equation}\label{strict}
\IF(V') \leq P\big( E(r(V'))\big)  < P(E)\,.
\end{equation}
Our conclusions now follow quickly. Indeed, if there exists an isoperimetric set $E$ of volume $|E|=V$, then inequality~(\ref{strict}) yields
\[
\IF(V') < P(E) = \IF(V)\,.
\]
This proves the second and the third claim of the theorem. Concerning the first one, for any $\eps>0$ we can take a set $E$ such that
\begin{align*}
|E| = V\,, && P(E) \leq \IF(V) +\eps\,.
\end{align*}
Hence, (\ref{strict}) tells us
\[
\IF(V') < P(E) \leq \IF(V) +\eps\,,
\]
for any $V'<V$, and since $\eps$ was arbitrary one gets the inequality $\IF(V')\leq \IF(V)$ for all $V'<V$.
\end{proof}

\section{Boundedness of isoperimetric sets\label{secbound}}

In this section we consider another property which seems reasonable. Assume that there exists some isoperimetric set: is it then obvious that it must be bounded? It appears reasonable that it should be so, at least when the density is increasing in the sense of Definition~\ref{defnondecr}. In fact, we can show that this is what always happens in the two-dimensional case (Proposition~\ref{propisobound2}), but that this is false in dimension $n\geq 3$ (Proposition~\ref{ex:nonex}). However, the result is true for any dimension if we consider an increasing density which is also radial (Theorem~\ref{thba}). The plan of the section is then the following: first we show the two-dimensional result, then we present our ``fundamental brick'' for the following constructions. Then, we show with two counterexamples, Propositions~\ref{ex:nonex} and~\ref{nonbdd}, that the two-dimensional result is too strong to hold in dimension $n$. Finally, we give the general $n-$dimensional results, Theorems~\ref{thba} and~\ref{lastfrank}.\bigskip

Let us state and prove our result concerning the two-dimensional case.

\begin{prop}[{\bf Boundedness in $\R^2$}\hspace{0pt}]\label{propisobound2}
Let $f$ be a density in $\R^2$ which either is non-decreasing (in the sense of Definition~\ref{defnondecr}) or approaches a finite limit $a>0$ at infinity. Then every isoperimetric set is bounded.
\end{prop}
\begin{proof}
Suppose that $E\subseteq \R^2$ is an isoperimetric set, and let $E_i$ be its closed connected components. First of all, we claim that every connected component $E_i$ is bounded: indeed, since $f$ is bounded below by some constant $c>0$ by assumption, 
\begin{equation}\label{per>diam}
+\infty>P_f(E_i) \geq c \PE (E_i) \geq c \, {\rm diam}\, E_i\,.
\end{equation}
As a consequence, we can assume that $E$ has infinitely many connected components; otherwise the claim is already proven. Moreover, each connected component is isolated by definition.\par
Fix now a connected component, say $E_1$. Making small variations of $E_1$ with respect to $u$, and keeping in mind formulae~(\ref{expansion-volume}) and~(\ref{expansion-perimeter}), we can consider suitable variations $E_1^\eps$ of $E_1$ for all $0<\eps<\bar\eps$ which do not intersect any of the other $E_i$'s and such that
\begin{align}\label{firstvarfirstcomp}
\big| E_1^\eps\big|_f = \big|E_1\big| + \eps\,, && P_f\big( E_1^\eps\big) \approx P_f \big( E_1\big) + \eps  H_f(E)
\leq P_f \big( E_1\big) + \eps \big(  H_f(E) + 1 \big)   \,,
\end{align}
where $H_f(E)$ is the generalized mean density of $E$, which is constant since $E$ is isoperimetric. Let us now consider another component $E_j$, and distinguish two cases.\par
The first case is when $f$ approaches a finite limit $a>0$ at infinity. It is admissible to assume that $E_j$ has distance at least $R$ from the origin, and that it has volume smaller than $\bar\eps$ (because $E$ has infinitely many connected components). This implies
\begin{align*}
\eps:=\big| E_j \big|_f \leq (a+\delta) \big| E_j\big|_{\rm eucl}\,, && P_f\big( E_j \big) \geq (a-\delta) \PE\big( E_j\big)\,,
\end{align*}
and by the smallness of the volume of $E_j$ we can assume
\begin{equation}\label{palso}
\big| E_j\big|_f \leq \frac{1}{H_f(E)+2}\, P_f(E_j)\,.
\end{equation}
Call now $\widetilde E$ the set we get from $E$ by removing $E_j$ and replacing $E_1$ by $E_1^\eps$: by construction $\big| \widetilde E\big|_f = \big| E\big|_f$, and thanks to~(\ref{firstvarfirstcomp}) and~(\ref{palso}) we also have
\[
P_f \big(\widetilde E\big) = P_f\big( E\big) + P_f\big( E_1^\eps \big) - P_f\big(E_1\big) - P_f\big( E_j\big)
\leq P_f\big( E\big) - \eps < P_f\big( E\big)\,,
\]
contradicting the minimality property of $E$.\par
The second case is when $f$ is nondecreasing. We will argue in a similar way as in the first case: writing for brevity $\rho e^{i\theta} = (\rho\cos \theta, \rho\sin\theta )$, for all relevant $\theta$ let
\begin{align*}
\rho^-(\theta) &:= \inf \big\{ \rho>0 :\, \rho e^{i\theta}\in E_j\big\}\,, &
\rho^+(\theta) &:= \sup \big\{ \rho>0 :\, \rho e^{i\theta}\in E_j\big\}\,,\\
\rho_m &:= \min \big\{ \rho^-(\theta)\big\}\,, &
\rho_M &:= \max \big\{ \rho^+(\theta)\big\}\,.
\end{align*}
We are then in the position of evaluating the perimeter and volume of $E_j$. Concerning the volume, one has
\begin{equation}\label{estvol}\begin{split}
\eps:&=\big| E_j \big|_f = \int \int_{\rho^-(\theta)}^{\rho^+(\theta)} f\big( \rho e^{i\theta}\big) \Chi{E_j}\big(\rho e^{i\theta} \big)\, \rho d\rho \, d\theta\\
&\leq \rho_M \int f\big( \rho^+(\theta) e^{i\theta}\big) \int_{\rho^-(\theta)}^{\rho^+(\theta)}  \Chi{E_j}\big(\rho e^{i\theta} \big)\,  d\rho \, d\theta
\leq \rho_M  L \int f\big( \rho^+(\theta) e^{i\theta}\big) d\theta\,,
\end{split}\end{equation}
where
\[
L := \max \int_{\rho^-(\theta)}^{\rho^+(\theta)}  \Chi{E_j}\big(\rho e^{i\theta} \big)\,  d\rho
\]
is the (weighted) length of the maximal radial slice of $E_j$. Concerning the perimeter, let
\[
\partial E_j^+ := \Big\{ \rho^+(\theta) e^{i\theta}:\, \rho^+(\theta)>0\Big\} \subseteq \partial E_j\,,
\]
and let $\H^1_f$ denote the Hausdorff one-dimensional measure on $\R^2$ with density $f$. We can then evaluate
\begin{equation}\label{estper}
P_f\big( E_j\big) = \H^1_f\big( \partial E_j\big)
\geq \H^1_f \big( \partial E_j^+ \big)
\geq \int f\big( \rho^+(\theta)e^{i\theta} \big) \rho^+(\theta)\, d\theta
\geq \rho_m \int f\big( \rho^+(\theta)e^{i\theta} \big) \, d\theta\,.
\end{equation}
Hence, putting together~(\ref{estvol}) and~(\ref{estper}), we find
\begin{equation}\label{firstvarjcomp}
P_f\big( E_j\big)  \geq \frac{1}{L}\,  \frac{\rho_m}{\rho_M}\,   \big| E_j \big|_f\,.
\end{equation}
Recall now that the volumes and diameters of the components $E_j$ go to 0. From this, we can assume that $\eps\leq \bar\eps$, as well as that
\begin{equation}\label{also}
\frac{1}{L}\, \frac{\rho_m}{\rho_M}> H_f(E)+2 \,.
\end{equation}
Consequently, if we call $\widetilde E$ the set we get from $E$ by removing $E_j$ and replacing $E_1$ by $E_1^\eps$, we have that $\big| \widetilde E\big|_f = \big| E\big|_f$, and thanks to~(\ref{firstvarfirstcomp}), (\ref{firstvarjcomp}) and~(\ref{also}) we also have
\[
P_f \big(\widetilde E\big) = P_f\big( E\big) + P_f\big( E_1^\eps \big) - P_f\big(E_1\big) - P_f\big( E_j\big)
\leq P_f\big( E\big) - \eps < P_f\big( E\big)\,,
\]
contradicting the assumption that $E$ is an isoperimetric set.
\end{proof}
Proposition~\ref{ex:nonex} will show with a counterexample that Proposition~\ref{propisobound2} is no longer true in dimension $n> 2$. The dimension $n=2$ played a crucial role in~(\ref{per>diam}) when we estimated the perimeter of a connected set by its diameter. Such an estimate fails in dimension $n \geq 3$, since thin tentacles can have large diameter and small perimeter.\bigskip

Let us now introduce the fundamental brick for our future examples of this section. We work in $\R^3$ for simplicity, but the same construction works for
any dimension $n\geq 3$. Let us use spherical coordinates $\rho,\, \phi,\, \theta$. This means that, for any $\rho\geq 0$, $0\leq \theta\leq 2\pi$, and $0 \leq \phi \leq \pi$ we denote by $(\rho,\,\theta,\,\phi)$ the point whose standard Euclidean coordinates are
\[
\Big( \rho \cos\phi,\, \rho\sin\phi \cos\theta,\, \rho\sin\phi\sin\theta \Big)\,.
\]
\begin{definition}[{\bf The fundamental brick}\hspace{0pt}]\label{funbri}
Let $R_1<R_2$ be positive big numbers, let $\phi_0>0$ be a small angle, and let $1\ll N\ll M\ll W$ be three big numbers. Let us define the set $E\subseteq \R^3$, as shown in Figure~\ref{Figure:E}: since our whole construction does not depend on $\theta$, the drawing shows just the plane where $\theta=0$ or $\theta=\pi$. The boundary of $E$ is given by
\[
\partial E:=\Big\{ \phi=\phi_0,\, R_1\leq \rho\leq R_2 \Big\} \cup \Big\{  0 \leq \phi \leq  \phi_0,\, \rho=  R_2 \Big\} \cup \Gamma\,,
\]
where $\Gamma$ is a spherical cap of radius $1$ containing the circle $\{\rho=R_1,\, \phi=\phi_0\}$. Let $\gamma: (0,\phi_0)\to \R^+$ be the function such that
\[
\Gamma=\big\{ (\rho,\theta,\phi):\, 0 < \phi < \phi_0,\, \rho=\gamma(\phi)\big\}\,.
\]
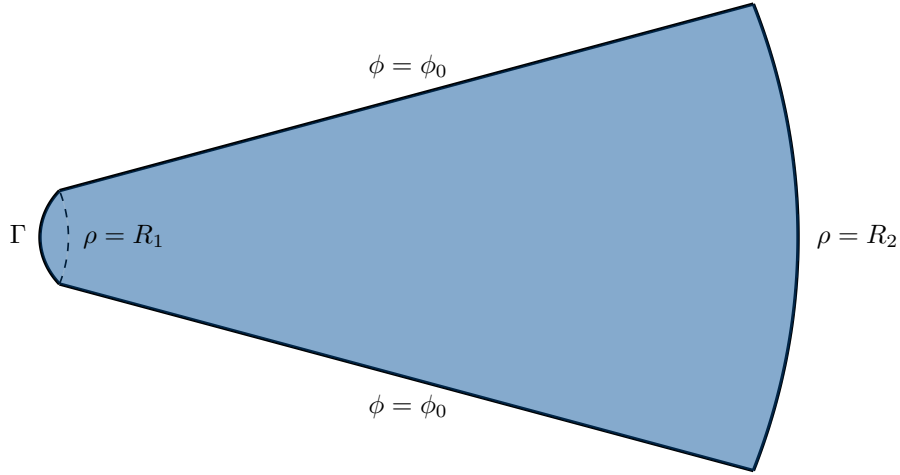
\begin{figure}[htbp]
\includegraphics[width=12cm]{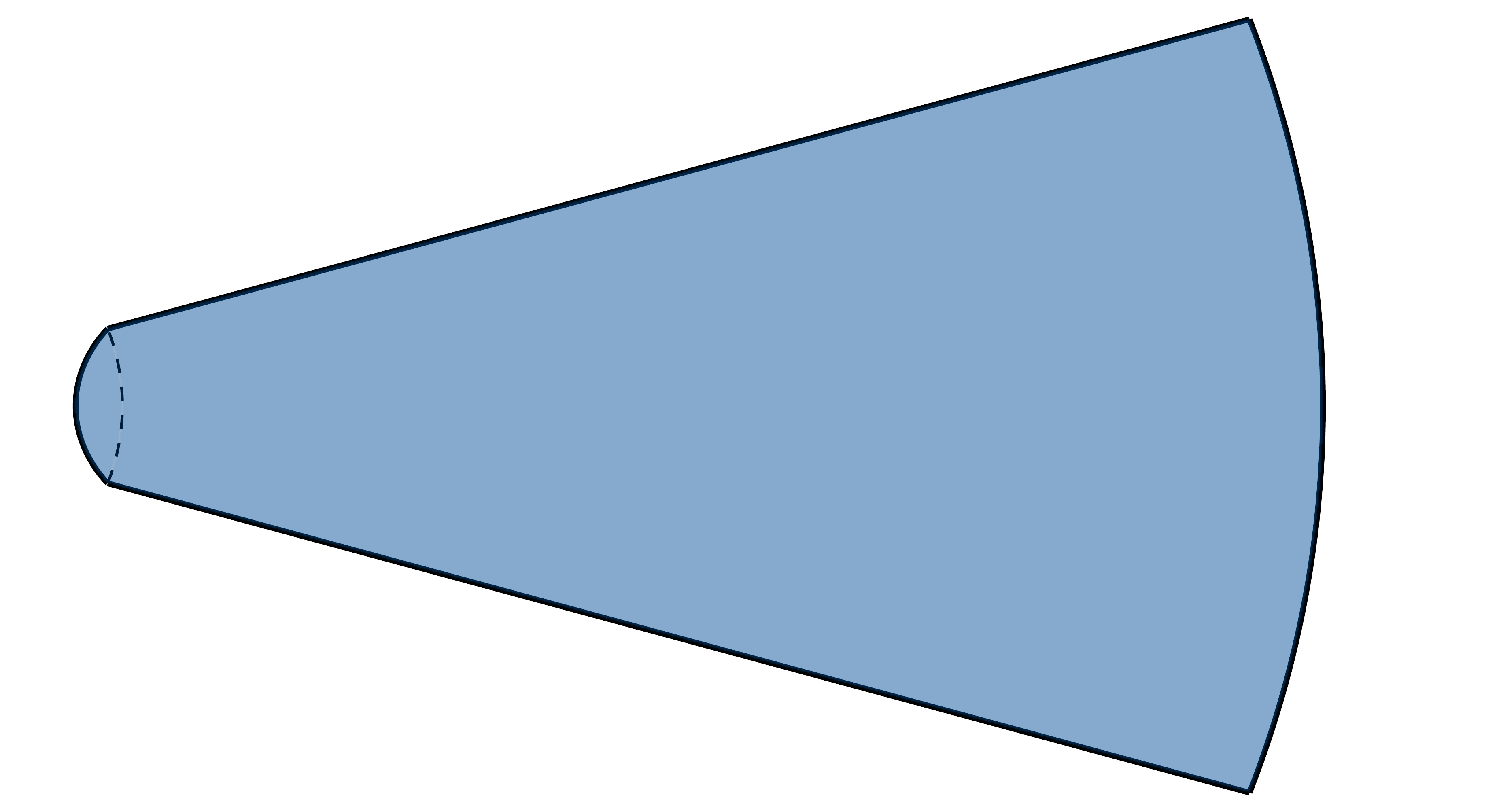}\caption{The fundamental brick $E$ for our examples of nonexistence and nonboundedness.}\label{Figure:E}
\setlength{\unitlength}{2000sp}
\begin{picture}(0,0)
\put(-4500,4200){\makebox(0,0)[lb]{\smash{{\SetFigFont{10}{12.0}{\rmdefault}{\mddefault}{\updefault}$\rho=R_1$}}}}
\put(4500,4200){\makebox(0,0)[lb]{\smash{{\SetFigFont{10}{12.0}{\rmdefault}{\mddefault}{\updefault}$\rho=R_2$}}}}
\put(-5400,4200){\makebox(0,0)[lb]{\smash{{\SetFigFont{10}{12.0}{\rmdefault}{\mddefault}{\updefault}$\Gamma$}}}}
\put(-1000,6300){\makebox(0,0)[lb]{\smash{{\SetFigFont{10}{12.0}{\rmdefault}{\mddefault}{\updefault}$\phi=\phi_0$}}}}
\put(-1000,2100){\makebox(0,0)[lb]{\smash{{\SetFigFont{10}{12.0}{\rmdefault}{\mddefault}{\updefault}$\phi=\phi_0$}}}}
\end{picture}
\end{figure}
Finally, we can define the density $f$ as
\[
f\big( \rho,\, \theta , \, \phi\big) := \left\{ \begin{array}{ll}
N \qquad & \hbox{if $\phi > \phi_0\,, \rho \leq R_1$}\,, \\
M \qquad & \hbox{if $\phi >\phi_0\,, R_1<\rho\leq R_2$}\,, \\
N & \hbox{if $\phi=\phi_0,\, \rho \leq R_2$}\,, \\
N & \hbox{if $\phi<\phi_0,\, \rho \leq \gamma(\phi)$}\,, \\
M & \hbox{if $\phi<\phi_0,\, \gamma(\phi)<\rho\leq R_2$}\,,\\
W & \hbox{if $\rho > R_2$}\,.
\end{array}\right.
\]
This density is non-decreasing and not radial. The set $E$ is clearly a good candidate to be isoperimetric, since it lies in a zone where the density has the huge value $M$, but most of its boundary lies where the density has the value $N\ll M$, namely, all the boundary except the part where $\rho=R_2$. Let us now fix a small number $\eps>0$ and set $\phi_0$, $R_1$ and $R_2$ so that
\begin{align}\label{defRRt}
R_1= \frac{1}{\eps^2}\,, && R_2 = R_1 + \frac{1}{\eps}\,, && R_2 \sin \phi_0 = \eps^{5/3}\,.
\end{align}
A simple evaluation of the perimeter and volume of $E$ gives
\begin{align}\label{evapervol}
\big| E \big|_f \approx M\, \eps^{7/3} \,, &&
P_f\big( E \big) \approx N\, \eps^{2/3} + M\, \eps^{10/3}\,.
\end{align}
\end{definition}
\bigskip

We are now already ready to prove the non-existence result.

\begin{prop}[{\bf Non-existence}\hspace{0pt}]\label{ex:nonex}
For $n\geq 3$, there exists a non-decreasing density on $\R^n$ such that general existence of isoperimetric sets fails. Indeed, $\IF(1)=0$.
\end{prop}
\begin{proof}
We work in $\R^3$ for simplicity; the same argument works in all higher dimensions as well. Let us take a sequence $\eps_j\searrow 0$: correspondingly, as described in Definition~\ref{funbri} we take sequences $R_{1,j}$, $R_{2,j}$ and $\phi_{0,j}$ fulfilling~(\ref{defRRt}) and we have the corresponding densities $f_j$ and sets $E_j$ in $\R^3$. Recalling~(\ref{evapervol}) we can choose $N_j$ and $M_j$ of order
\begin{align*}
N_j \approx \eps_j^{-1/3}\,, && M_j \approx \eps_j^{-7/3}\,,
\end{align*}
in such a way that
\begin{align*}
\big| E_j \big|_{f_j} =1 \,, &&
P_{f_j}\big( E_j \big) \approx \eps_j^{3} + \eps_j^{1/3} + \eps_j \to 0\,.
\end{align*}
Let $W_j=N_{j+1}$. Provided $\eps_j$ is chosen sufficiently rapidly decreasing to $0$, we have
\begin{align*}
R_{2,j} \ll R_{1,j+1}\,, && M_j \ll W_j = N_{j+1}\,.
\end{align*}
Finally, we can define a single density $f$ on $\R^3$ just setting
\[
f \equiv f_j \qquad \hbox{on } B\big(R_{2,j}\big) \setminus B\big(R_{2,j-1}\big)\,,
\]
so that $f$ is non-decreasing and diverging. In this way, all the sets $E_j$ are disconnected (and very far from each other), and the density around each $E_j$ coincides with $f_j$. As a consequence, all the sets $E_j$ have unit volume, but their perimeters are going to $0$. In particular, $\IF(1)=0$ and there is no isoperimetric set of volume $1$.
\end{proof}

Our next goal is to show an example of non-boundedness of an isoperimetric set. This cannot be done, of course, with the very same construction as in the preceding example, since in that case there was no isoperimetric set at all! We will present a slight modification of the argument, where the isoperimetric set will intersect all the sets $E_j$. To show its optimality, we will need to use the following Propositions~\ref{super-lemma} and~\ref{super-lemmabis}.
\begin{prop}\label{super-lemma}
Let $\eps$ be a small positive number, let $\phi_0$, $R_1$, $R_2$ be according to~(\ref{defRRt}), and let $W\gg M\gg N\gg 1$ be so that the corresponding fundamental brick $E$ and density $f$ of Definition~\ref{funbri} satisfy $|E|_f\leq 1$. There exists a universal constant $c>0$ such that, for any set $F\subseteq \R^3$ with $|F|\leq 1$, one has
\begin{equation}\label{claim}
P(F) - P(F\cap E) \geq c \big|F\setminus E\big| \,.
\end{equation}
\end{prop}
The proof will use the following two lemmas, the first of which we state without proof.
\begin{lemma}\label{lipschitz} If $C\subseteq \R^n$ is a convex set, then the projection $\pi: \R^n \setminus C \to C$, which associates to any $x\notin C$ the closest point $\pi(x)\in C$, is $1-$Lipschitz.
\end{lemma}
\begin{lemma}\label{outofball}
Let $B(r)$ be a ball of radius $r$ and $S(r)$ its boundary, and let $F\subseteq \R^n\setminus B(r)$ be a set of volume $\VE{F} \leq 1$. Then there exists a constant $c(r)$, depending on $r$ but not on $F$, such that
\[
\VE{F}  \leq c(r) \Big( \H^{n-1}\big( \partial F \setminus S(r)\big) - \H^{n-1}\big( \partial F \cap S(r)\big) \Big)\,.
\]
\end{lemma}
\begin{proof}
Let $F$ be as in the claim, and define $G=F\cup B(r)$. By the isoperimetric inequality,
\[\begin{split}
n\omega_n^{\frac 1n} \Big( \VE{G} \Big)^{\frac{n-1}{n}} 
&\leq \PE \big( G \big)
= \PE ( F ) + \PE \big(B(r)\big) - 2 \H^{n-1} \big( \partial F \cap S(r)\big)\\
&= n\omega_n r^{n-1}
+\H^{n-1}\big( \partial F \setminus S(r) \big)   -  \H^{n-1} \big( \partial F \cap S(r)\big)\,.
\end{split}\]
Hence,
\[\begin{split}
\H^{n-1}\big( \partial F \setminus S(r) \big)   &-  \H^{n-1} \big( \partial F \cap S(r)\big) \geq 
n\omega_n^{\frac 1n} \Big( \VE{G} \Big)^{\frac{n-1}{n}} - n\omega_n r^{n-1}\\
&=n\omega_n^{\frac 1n} \Big(\omega_n r^n+ \VE{F} \Big)^{\frac{n-1}{n}} - n\omega_n r^{n-1}
\geq  (n-1) \bigg(  r^n + \frac 1{\omega_n}\bigg)^{-\frac 1n} \,\VE{F}\,,
\end{split}\]
so that the result follows by taking
\[
c(r) = \frac 1 {n-1}\,\bigg(  r^n + \frac 1{\omega_n}\bigg)^{\frac 1n}\,.
\]
\end{proof}
Having these two lemmas in hand, we are ready to prove Proposition~\ref{super-lemma}.
\begin{figure}[htbp]
\includegraphics[width=8cm]{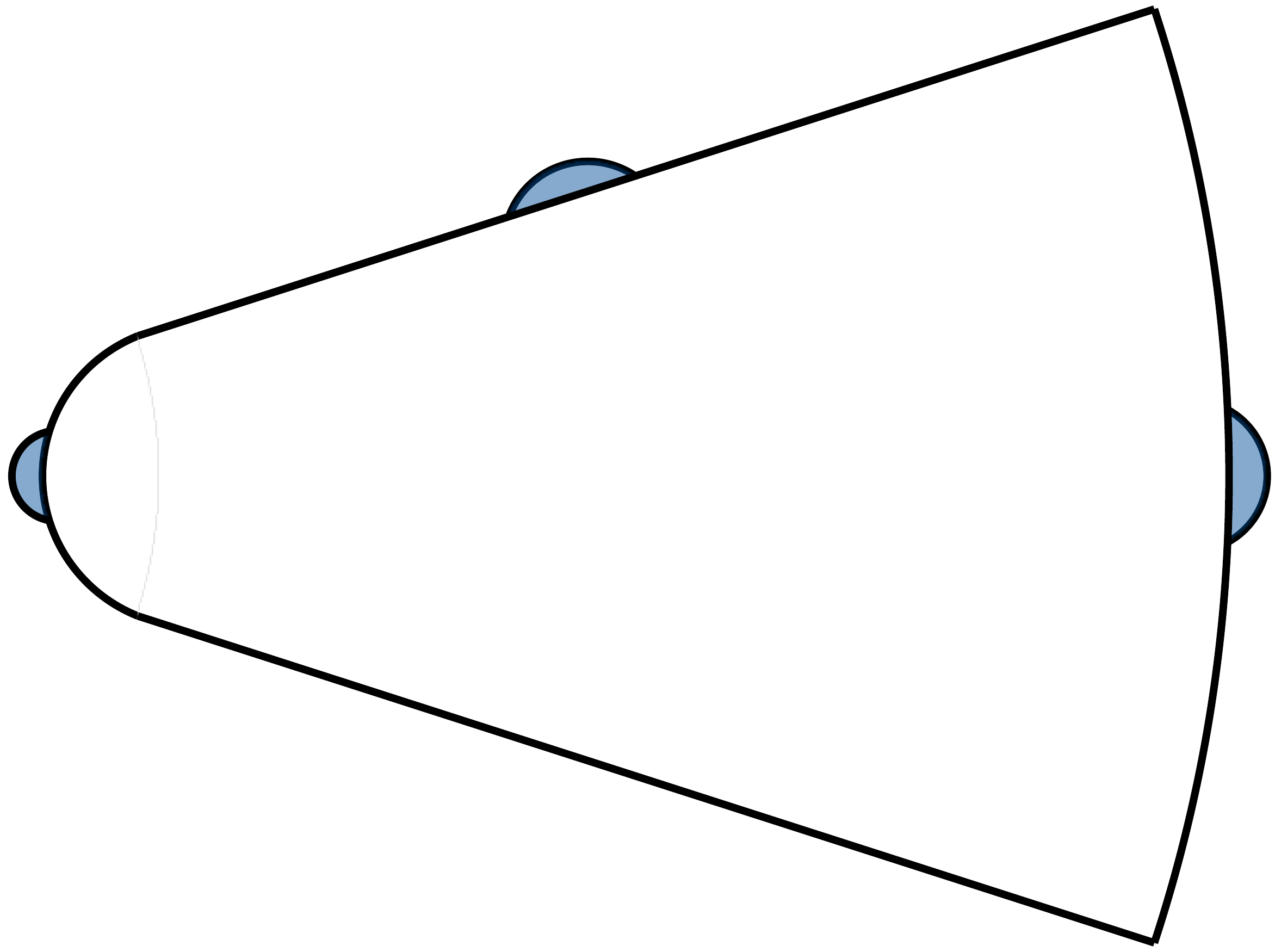}\caption{The sets $F_1$, $F_2$ and $F_3$ for Proposition~\ref{super-lemma}.}\label{Figure:EF}
\setlength{\unitlength}{2000sp}
\begin{picture}(0,0)
\put(-4200,4000){\makebox(0,0)[lb]{\smash{{\SetFigFont{10}{12.0}{\rmdefault}{\mddefault}{\updefault}$F_1$}}}}
\put(-1000,6000){\makebox(0,0)[lb]{\smash{{\SetFigFont{10}{12.0}{\rmdefault}{\mddefault}{\updefault}$F_2$}}}}
\put(3800,4000){\makebox(0,0)[lb]{\smash{{\SetFigFont{10}{12.0}{\rmdefault}{\mddefault}{\updefault}$F_3$}}}}
\end{picture}
\end{figure}
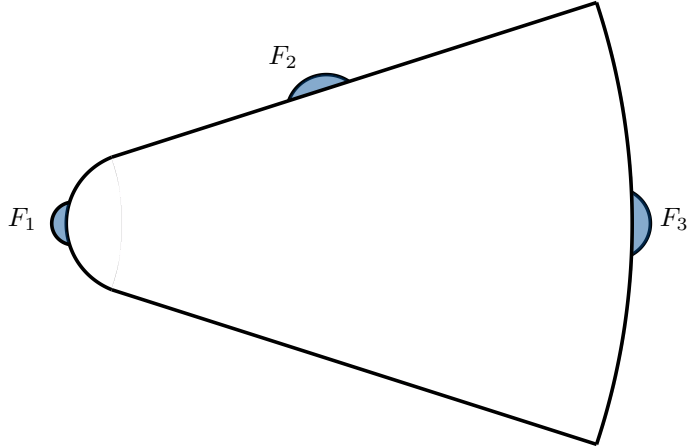
\begin{proof}[Proof of Proposition~\ref{super-lemma}]
Given a set $F\subseteq \R^3$, let $F^+=F\setminus \overline{E}$, and write $F^+=F_1\cup F_2 \cup F_3$, where
\begin{align*}
F_1 = F^+ \cap B(R_1)\,, && F_2 = F^+\cap \Big( B(R_2)\setminus B(R_1)\Big)\,, && F_3 = F^+ \setminus B(R_2)\,.
\end{align*}
For simplicity, we start assuming that the closures of the three sets are disjoint. The situation is depicted in Figure~\ref{Figure:EF}. Defining, for $i=1,\,2,\,3$,
\begin{align*}
\partial F_i^+ = \partial F_i \setminus \partial E\,, &&
\partial F_i^- = \partial F_i \cap \partial E\,,
\end{align*}
one readily observes that
\begin{equation}\label{mainhere}
P(F)- P\big(F\cap E\big) \geq \H^{2}_f \Big( \partial F_1^+ \cup \partial F_2^+ \cup \partial F_3^+ \Big)
-\H^{2}_f \Big( \partial F_1^-+ \cup \partial F_2^- \cup \partial F_3^- \Big)\,.
\end{equation}
Consider now $F_1$: Lemma~\ref{outofball} immediately tells us that
\begin{equation}\label{estF1}\begin{split}
\big|F_1\big| &= N \VE{F_1}  \leq N c \Big( \H^{2}\big( \partial F_1^+\big) - \H^{2}\big( \partial F_1^-\big) \Big)\\
&=  c \Big( \H^{2}_f\big( \partial F_1^+\big) - \H^{2}_f\big( \partial F_1^-\big) \Big)\,.
\end{split}\end{equation}
Let us then study $F_2$. Since $|F_2|\leq 1$, the isoperimetric inequality tells us that
\begin{equation}\label{eqF2eucl}
\PE\big( F_2 \big) \geq c \VE{F_2}^{\frac 23} \,.
\end{equation}
Moreover, since the projection of $\partial F_2^+$ on $E$ contains $\partial F_2^-$, by Lemma~\ref{lipschitz} we have
\[
\H^{2} \big(F_2^+ \big) \geq \H^2\big(F_2^-\big)\,;
\]
hence by~(\ref{eqF2eucl}) we get
\begin{equation}\label{estF2}\begin{split}
\H^2_f\big( \partial F_2^+\big) - \H^2_f\big( \partial F_2^-\big) &=
M \H^2\big( \partial F_2^+\big) - N \H^2\big( \partial F_2^-\big)
\geq (M-N ) \H^2 \big(\partial F_2^+\big)\\
&\geq \frac{M-N}{2}\, \PE\big( F_2 \big)
\geq c\,\frac{M-N}{2}\, \VE{F_2}^{\frac 23}
\geq c\,\frac{M-N}{2 M^{\frac 23}}\, \big| F_2 \big|_f\,.
\end{split}\end{equation}
The very same argument works also for $F_3$; hence we also have
\begin{equation}\label{estF3}
\H^2_f\big( \partial F_3^+\big) - \H^2_f\big( \partial F_3^-\big) \geq c\,\frac{W-M}{2 W^{\frac 23}}\, \big| F_3 \big|_f\,.
\end{equation}
Inserting~(\ref{estF1}), (\ref{estF2}) and~(\ref{estF3}) into~(\ref{mainhere}), and recalling that $1\ll N \ll M\ll W$, we have
\[
P(F) - P(F\cap E) \geq c \big| F^+ \big|\,,
\]
that is, (\ref{claim}).\par\bigskip
We now consider the general case, where $F_1$, $F_2$ and $F_3$ might have common boundary. In this case, of course~(\ref{estF1}), (\ref{estF2}) and~(\ref{estF3}) still hold, but~(\ref{mainhere}) is no longer true and a little more care is needed. Let us define
\begin{align*}
J := \partial F_1^+ \cap \partial F_2^+\,, && K := \partial F_2^+ \cap \partial F_3^+\,,
\end{align*}
so that~(\ref{mainhere}) becomes
\begin{equation}\label{newmainhere}\begin{split}
P(F)&- P\big(F\cap E\big) \geq \\
&\H^{2}_f \Big( \partial F_1^+ \cup \partial F_2^+ \cup \partial F_3^+ \Big)
-\H^{2}_f \Big( \partial F_1^-+ \cup \partial F_2^- \cup \partial F_3^- \Big)   - 2 \H^{2}_f (J) - 2\H^2_f(K)\,.
\end{split}\end{equation}
Observe now that the projection of $\partial F_2^+\setminus J$ on $\partial B(R_1)$ (resp. $E$) contains $J$ (resp. $\partial F_2^-$); hence
\begin{align*}
\H^{2} \big( \partial F_2^+ \setminus J\big) \geq \H^2 (J) \,, &&
\H^{2} \big( \partial F_2^+ \setminus J\big) \geq \H^2 \big(\partial F_2^-\big) \,.
\end{align*}
As a consequence, (\ref{estF2}) can be improved to
\begin{equation}\label{newestF2}\begin{split}
\H^2_f \big(\partial F_2^+\big) - \H^2_f\big( \partial F_2^-\big) - 2 \H^2_f(J)
&= M \H^2 \big(\partial F_2^+\setminus J\big) - N \H^2\big( \partial F_2^-\big) - N \H^2(J)\\
&\geq \big(M-2N\big) \H^2 \big(\partial F_2^+\setminus J\big)
\geq \frac{M-2N}{3} \, \PE \big(F_2\big)\\
&\geq c\, \frac{M-2N}{3M^{\frac 23}} \, \big| F_2 \big|_f\,.
\end{split}\end{equation}
In the very same way, (\ref{estF3}) becomes
\begin{equation}\label{newestF3}
\H^2_f \big(\partial F_3^+\big) - \H^2_f\big( \partial F_3^-\big) - 2 \H^2_f(K)
\geq c\, \frac{W-2M}{3W^{\frac 23}} \, \big| F_2 \big|_f\,.
\end{equation}
Hence, inserting~(\ref{estF1}), (\ref{newestF2}) and~(\ref{newestF3}) into~(\ref{newmainhere}), we get the general validity of~(\ref{claim}). Notice that the constant $c$ coincides with $c(1)$ in Lemma~\ref{outofball}, hence it does not depend on the choice of $R_1$, $R_2$, $\phi_0$, $N$, $M$ or $W$.
\end{proof}
\begin{prop}\label{super-lemmabis}
Consider a set $E$ and a density $f$ as in Definition~\ref{funbri}, and suppose that
\begin{equation}\label{smallperimeter}
P_f(E)\leq \frac {c}{4} \, |E|\,,
\end{equation}
where $c$ is the constant of Proposition~\ref{super-lemma}. There exists a set $G\subseteq E$ such that $|G|\geq |E|/2$ and that, for any other $F\subseteq E$, one has
\begin{equation}\label{belowsegment}
P_f(F) - P_f(G)  \geq  \frac c 2 \, \Big( |F |_f - | G |_f\Big) \,.
\end{equation}
Moreover, the inequality is strict for any $F\subseteq E$ such that $|F|<|E|/2$.
\end{prop}
\begin{proof}
By lower semicontinuity of the density and by compactness, for any $0< V \leq |E|$ there exists a set $F_V\subseteq E$ minimizing the perimeter among all the subsets of $E$ with volume $V$.\par
Let us then consider the continuous function $\IF_E : \big[0, |E|\big]\to \R^+$ defined by $\IF_E(V) = P\big(F_V\big)$, let
\[
\lambda : = \min \Big\{ \IF_E(V) - \frac c2 \,V:\, 0\leq V \leq |E|\Big\}\,,
\] 
and let
\[
\overline V := \max \Big\{ 0\leq V \leq |E| :\, \IF_E(V) - \frac c2\, V  = \lambda \Big\}\,.
\]
We claim that $G := F_{\overline V}$ is as required.\par
First of all, notice that $\overline V \geq |E|/2$, due to the fact that for any $V < |E|/2$, by~(\ref{smallperimeter}) we get
\begin{equation}\label{>1/2}
\lambda \leq \IF_E \big(|E| \big) - \frac c2\, |E|  = P(E) - \frac c2\, |E| \leq -\frac c4 \, |E| 
\leq \IF_E(V) -\frac c4 \, |E| 
< \IF_E(V) -\frac c2 \, V\,.
\end{equation}
We only have then to check that~(\ref{belowsegment}) holds. To this aim, take any set $F\subseteq E$ and suppose without loss of generality that $F=F_V$ for some $V$. Hence, to show~(\ref{belowsegment}) one just has to notice that
\[
P_f(F) - P_f(G)  \geq  \frac c 2 \, \Big( |F |_f - | G |_f\Big)  \qquad \Longleftrightarrow \qquad
\IF_E(V) - \IF_E(\overline V)  \geq  \frac c 2 \, \big( V - \overline V\big)\,,
\]
which is in turn true by the definition of $\overline V$. To conclude, we observe that inequality~(\ref{belowsegment}) holds strictly whenever $|F|<|E|/2$ thanks to~(\ref{>1/2}).
\end{proof}

We are finally ready to show our example of a non-bounded isoperimetric set. Notice that the Proposition~\ref{nonbdd} below is false for $n=2$ thanks to Proposition~\ref{propisobound2}, and it is also trivially false for $n=1$.

\begin{prop}[{\bf Non-boundedness}\hspace{0pt}]\label{nonbdd}
For each $n\geq 3$, there exists a non-decreasing density on $\R^n$ and a volume $V$ such that there exist isoperimetric sets of volume $V$, but none of them is bounded.
\end{prop}
\begin{proof}
We treat the case $n=3$; the other cases are the same. We will divide the proof into three steps.
\step{I}{The geometrical setting.}
We will again use the fundamental brick $E$ of Definition~\ref{funbri} and Figure~\ref{Figure:E}. As in Proposition~\ref{ex:nonex}, we take a rapidly decreasing sequence $\eps_j\searrow 0$ and sequences $R_{1,j}$, $R_{2,j}$ and $\phi_{0,j}$ fulfilling~(\ref{defRRt}). This defines correspondingly densities $f_j$ and sets $E_j\subseteq \R^3$, so having~(\ref{evapervol}) in mind we take $N_j$ and $M_j$ of order
\begin{align*}
N_j \approx \frac{\eps_j^{-1/3}}{2^j}\,, && M_j \approx \frac{\eps_j^{-7/3}}{2^j}\,,
\end{align*}
in such a way that
\begin{align}\label{pervol2^j}
\big| E_j \big|_{f_j} =\,\frac{1}{2^j} \,, &&
\frac{P_{f_j}\big( E_j \big)}{\big| E_j \big|_{f_j}} \approx  \eps_j^{3} + \eps_j^{1/3} + \eps_j \xrightarrow[\,j\to\infty\,]{} 0\,.
\end{align}
Defining $W_j=N_{j+1}$, we observe that
\begin{align*}
R_{2,j} \ll R_{1,j+1}\,, && M_j \ll W_j = N_{j+1}\,,
\end{align*}
and we define the density
\[
f \equiv f_j \qquad \hbox{on } B\big(R_{2,j}\big) \setminus B\big(R_{2,j-1}\big)
\]
on $\R^3$, which is nondecreasing and diverging.\par
By~(\ref{pervol2^j}), it is admissible to assume that
\[
P_f(E_j) <  \frac c4\, \big| E_j \big|_f
\]
for all $j$; hence by Proposition~\ref{super-lemmabis} we can also define the sets $G_j$ having volume $V_j\geq |E_j|/2$ in such a way that~(\ref{belowsegment}) holds true. Here, and in the rest of the proof, $c$ is the constant of Proposition~\ref{super-lemma}. Finally, we set $E:=\cup_j E_j$ and $G:= \cup_j G_j$. We will show that $G$, which by construction is not bounded, is an isoperimetric set of volume
\[
V := \sum_j V_j \in \bigg[\, \frac 12 \,, 1 \, \bigg]\,.
\]
\step{II}{Removing the external part.}
Let us take a generic set $F$ of volume $|F| = V$, and call $F^{ext}= F \setminus E$. Our goal in this step is to show that
\begin{equation}\label{weakclaim}
P(F) - P(F\cap E) \geq \frac c2 \, \big| F \setminus E \big|\,,
\end{equation}
that is, (\ref{claim}) still holds with a worse constant. To this aim, for any $j\in \N$ let us consider the region
\[
F \cap B\bigg( \frac{R_{2,j}+R_{1,j+1}}2 + \frac 8c \bigg) \setminus B\bigg( \frac{R_{2,j}+R_{1,j+1}}2  - \frac 8c\bigg)\,,
\]
take
\[
\frac{R_{2,j}+R_{1,j+1}}2 - \frac 8c \leq \rho_j \leq \frac{R_{2,j}+R_{1,j+1}}2 + \frac 8c
\]
to minimize $\H^2\big( F \cap S(\rho)\big)$, and define
\[
F_j = F^{ext} \cap B\big( \rho_ j \big) \setminus B\big( \rho_{j-1}\big)\,.
\]
As in Proposition~\ref{super-lemma}, for any $j$, let
\begin{align*}
\partial F_j^+ = \partial F_j \setminus \partial E_j\,, &&
\partial F_j^- = \partial F_j \cap \partial E_j\,.
\end{align*}
We can observe that
\begin{equation}\label{toinsert}
P(F) - P(F\cap E) \geq \sum_j \H^2_f\big( \partial F_j^+ \big) - \H^2_f \big( \partial F_j^-\big) - 2 \H^2_f\Big( F^{ext} \cap S\big(\rho_j\big)\Big)\,,
\end{equation}
and Proposition~\ref{super-lemma} ensures us that for any $j$
\begin{equation}\label{recall}
\H^2_f(\partial F_j^+) - \H^2_f(\partial F_j^-) \geq c \big| F_j \big|\,.
\end{equation}
Let us now consider more closely the slice $F^{ext} \cap S\big(\rho_j\big)$. If one has that
\[
\rho_j \leq \frac{R_{2,j} + R_{1,j+1}}{2}\,,
\]
then by definition we have
\[
\big| F_{j+1} \big| \geq \frac 8c \, \H^2_f \Big( F^{ext}\cap S\big( \rho_j\big)\Big)\,.
\]
On the other hand, if
\[
\rho_j \geq \frac{R_{2,j} + R_{1,j+1}}{2}\,,
\]
then for the same reason we get
\[
\big| F_j \big| \geq \frac 8c \, \H^2_f \Big( F^{ext}\cap S\big( \rho_j\big)\Big)\,.
\]
In any case, we can then conclude that
\[
\H^2_f \Big( F^{ext}\cap S\big( \rho_j\big)\Big) \leq  \frac c8\, \Big( \big| F_j \big| + \big| F_{j+1} \big|\Big)\,,
\]
and adding up this implies that
\[
2 \sum_j \H^2_f \Big( F^{ext}\cap S\big( \rho_j\big)\Big) \leq  \frac c2 \, \sum_j \big| F_j\big|\,.
\]
Inserting this last inequality into~(\ref{toinsert}) and recalling~(\ref{recall}) for the last time, we get~(\ref{weakclaim}).

\step{III}{Conclusion.}
We are now ready to conclude the proof. Take a general set $F$ with $|F|=V$. By Step~II we know that~(\ref{weakclaim}) holds. Let $F^{int} = F \cap E$, and write $F^{int} = \cup_j F_j$ where
\[
F_j := F^{int} \cap E_j\,.
\]
Thanks to~(\ref{belowsegment}), we know that
\begin{equation}\label{last}
P_f\big(F_j \big) - P_f\big(G_j \big)  \geq  \frac c 2 \, \Big( \big|F_j \big|_f - \big| G_j \big|_f\Big)
\end{equation}
for all $j$. If we add up this inequality for all $j$, we find
\[
P_f\big(F^{int} \big) - P_f \big(G\big)  \geq  \frac c 2 \, \Big( \big|F^{int} \big |_f - \big| G \big|_f\Big)\,,
\]
and adding this to~(\ref{weakclaim}) we get
\[
P_f(F) - P_f \big(G\big)  \geq  \frac c 2 \, \Big( \big|F^{int} \big |_f - \big| G \big|_f + \big| F \setminus E\big|_f\Big)=0 \,.
\]
Since $F$ was arbitrary, we have finally proved that $G$ is an isoperimetric set of volume $V$, and it is unbounded by construction. It remains to be shown that \emph{any} isoperimetric set of volume $V$ is unbounded. Suppose then that $F$ is an isoperimetric set of volume $V$, so that all the above inequalities are equalities. The fact that~(\ref{last}) is an equality implies, by Proposition~\ref{super-lemmabis}, that $|F_j| \geq |E_j|/2$. This tells us that $F$ intersects each $E_j$, and in turn this ensures that $F$ is not bounded.
\end{proof}

We can now show the result about the boundedness of an isoperimetric set in dimension $n\geq 3$. This result is weaker than the two-dimensional result but, thanks to Proposition~\ref{nonbdd}, it is also sharp.
\begin{theorem}[{\bf Boundedness in $\R^n$, I}\hspace{0pt}]\label{thba}
Consider a radial, non-decreasing density $f$ on $\R^n$. Then every isoperimetric set is bounded.
\end{theorem}
\begin{proof}
Let $E$ be an isoperimetric set and assume that it is not bounded. Let
\begin{align*}
E(r) &:= E \cap B(r) \subsetneq E\,, & E_r&:= E \cap S(r) \,, \\
P(r) &:= \H^{n-1}_f\big( \partial E \setminus B(r)\big)\,, & V(r)& :=\H^n_f\big( E \setminus B(r)\big)\,;
\end{align*}
that is, $E(r)$ is the part of $E$ inside the ball $B(r)$, $E_r$ is the slice of $E$ at distance $r$ from the origin, and $P(r)$ and $V(r)$ are the perimeter and the volume of $E$ outside of the ball $B(r)$. Recall that, with the same assumptions and notation, in Theorem~\ref{Lambda-increasing} we proved that $P\big( E(r)\big) < P (E)$, in equation~(\ref{perimeter-decreases}). Since
\[
P\big(E(r)\big) = P(E) - P(r) + \H^{n-1}_f\big(E_r\big)\,,
\]
the last estimate can be rewritten as
\begin{equation}\label{oldestimate}
P(r) > \H^{n-1}_f\big(E_r\big)\,.
\end{equation}
Considering only radii $r\geq 1$, the standard isoperimetric inequality in the sphere tells us that for any subset $E_r$ of the sphere $S(r)$ having area at most half of the sphere, one has
\[
\H^{n-2}\big( \partial E_r \big)  \geq c \Big( \H^{n-1} \big( E_r \big) \Big)^{\frac{n-2}{n-1}}\,,
\]
where $\partial E_r$ denotes the boundary of $E_r$ inside $S(r)$. And in turn, if $E$ has bounded perimeter, than $\H^{n-1}\big(E_r\big)\leq \frac 12 \H^{n-1}\big(S(r)\big)$ for all $r$ big enough. Recalling that in $S(r)$ the density has the constant value $f(r)$, the last inequality is equivalent to
\[
\H^{n-2}_f\big( \partial E_r \big) \geq c \Big( \H^{n-1}_f \big( E_r \big) \Big)^{\frac{n-2}{n-1}}\, f(r)^{\frac{1}{n-1}}\,,
\]
which in turn by~(\ref{oldestimate}) leads to
\begin{equation}\label{semioldestimate}
\H^{n-2}_f\big( \partial E_r \big) \geq c P( r )^{\frac{-1}{n-1}} \H^{n-1}_f \big( E_r \big)\,,
\end{equation}
where we have also used the fact that $f$ is bounded from below (notice that the positive constant $c$ may decrease from line to line).\par
Now, observe that
\[
-\frac{\partial P(r)}{\partial r} = \bigg| \frac{\partial P(r)}{\partial r} \bigg| \geq \H^{n-2}_f\big( \partial E_r\big)\,,
\]
and that
\[
\H^{n-1}_f(E_r) = - \frac{\partial V}{\partial r} \,(r)\,,
\]
hence~(\ref{semioldestimate}) can be further rewritten as
\[
-\frac{\partial}{\partial r}\Big( P( r )^{\frac{n}{n-1}}\Big) \geq  - c \, \frac{\partial}{\partial r} \Big( V(r)\Big)\,.
\]
Recalling that both $P(r)$ and $V(r)$ converge to $0$ when $r$ goes to $+\infty$, an integration over $r$ yields
\[
P( r )^{\frac{n}{n-1}} \geq  c V(r)\,.
\]
Arguing exactly as in~(\ref{firstvarfirstcomp}), we can pick $R\in \R$ such that $E\cap B(R)\neq \emptyset$ and observe that for $0<\eps<\bar\eps$ it is possible to define a set $E_\eps$ such that
\begin{align*}
E_\eps \setminus B(R)  = E \setminus B(R)\,, &&
\big| E_\eps\big| = \big| E\big| +\eps\,, &&
P\big( E_\eps \big) \leq P(E) + \eps \big(H(E) +1\big)\,.
\end{align*}
If we take then $r>R$ such that $\eps = V(r) < \bar \eps$, and we define $\widetilde E = E_\eps \cap B(r)$, then of course $\big|\widetilde E\big| = |E|$, and moreover
\[
P\big( \widetilde E\big) = P \big( E_\eps\big) - P(r) + \H^{n-1}_f(E_r)
\leq P(E) + \eps\big(H(E) +1\big)- c \,\eps^{\frac{n-1}{n}} + \H^{n-1}_f(E_r)\,.
\]
Since we can have $\eps$ arbitrarily small up to take $r\gg 1$, and since $E$ is an isoperimetric set, we deduce that for all $r\gg 1$ one has
\[
\H^{n-1}_f(E_r) \geq c \,\eps^{\frac{n-1}{n}}\,,
\]
that is,
\[
- \frac{\partial V}{\partial r} \,(r) \geq c\, V(r)^{\frac{n-1}{n}}\,,
\]
or, equivalently,
\[
- \frac{\partial}{\partial r}\Big( V(r)^{\frac 1n}\Big) \geq c\,.
\]
This last estimate gives a contradiction with the assumption $V(r)>0$ for all $r$, and the proof is complete.
\end{proof}

We show now with our last result about the boundedness of isoperimetric sets.
\begin{theorem}[{\bf Boundedness in $\R^n$, II}\hspace{0pt}]\label{lastfrank}
Let $f$ be a ${\rm C}^1$ density on $\R^n$ such that $|D f| \leq  Cf$ for some constant $C>0$. Then every isoperimetric set is bounded.
\end{theorem}
\begin{proof}
Let $E$ be an isoperimetric set. Then, the generalized mean curvature $H_f$ is constant on $\partial E$ (see Definition~\ref{X.x}). Recall by~(\ref{X.x'}) that, writing $f=e^v$, one has by definition
\[
H_f(x,E) = H_0 (x,E) + \frac{\partial v}{\partial \nu_E(x)} (x)
\]
for all $x\in\partial E$ such that the normal $\nu_E(x)$ to $\partial E$ at $x$ exists. The assumption $|D f|\leq C f$ ensures that the second term on the right in~(\ref{X.x'}) is bounded. Therefore, the Euclidean mean curvature $H_0$ of $\partial E$ also is bounded. As a consequence, the perimeter of $E$ inside a unit ball centered at points of $\partial E$ is bounded by below by a strictly positive constant: this comes from the so-called \emph{monotonicity of the mass ratio}, see~\cite[Chapt. 9 and 11.2]{M1}.
Since the set has finite perimeter, it readily follows that it is bounded.
\end{proof}

Let us conclude by collecting in a single corollary our three boundedness results, namely, Proposition~\ref{propisobound2}, Theorem~\ref{thba} and Theorem~\ref{lastfrank}.

\begin{corollary}\label{corbdd}
Let $E$ be an isoperimetric set in $\R^n$ with density $f$. Then E is bounded if any of the following three hypotheses hold:
\begin{enumerate}
\item $n=2$ and $f$ is increasing, or
\item $f$ is radial and increasing, or
\item $f$ is ${\rm C}^1$ and $|Df|\leq Cf$.
\end{enumerate}
\end{corollary}

\section{Geometric properties of isoperimetric sets\label{sect5}}

In this section, we want to discuss some geometric properties of isoperimetric sets, namely, whether or not they are radial, or convex, and whether or not they contain the shortest paths between their points. We point out that the results of Section~\ref{seccon} do not depend on those of the present section.

We start with spherical symmetrication (cf. \cite[Sect. 9.2]{BZ}, \cite[Remark 4]{MHH}).
\begin{definition}
Let $E\subseteq \R^n$. For any $\rho>0$, define $A_E(\rho)$ as the area of the section $E\cap S(\rho)$. Define the \emph{spherical symmetrization of $E$} the set $E^*\subseteq \R^n$such that $A_{E^*} \equiv A_E$, and such that $E^*\cap S(\rho)$ is a spherical cap centered at $(\rho,0, \dots ,0)$.
\end{definition}

The operation of symmetrization has the great advantage of simplifying the set, while maintaining volume and reducing perimeter. The proof of this fact is quite similar to the standard proof for Steiner symmetrization (see~\cite{K}), so we will just give a sketch for it.
\begin{theorem}\label{thmSteiner}
Let $f$ be a radial density on $\R^n$ and let $E$ be a set of finite volume. Then the spherical symmetrization $E^*$ satisfies
\begin{align}\label{Steiner}
\big| E^* \big| = \big| E \big|\,, && P\big( E^*\big) \leq P(E)\,.
\end{align}
Suppose further that $E$ is an open set of finite perimeter, and let $\nu(x)$ denote the normal vector at any $x\in\partial E$. If
\begin{equation}\label{rlp}
\H^{n-1}\bigg( x\in \partial E:\, \nu(x) = \pm \frac x{|x|}\bigg) =0\,,
\end{equation}
and the set
\[
I_E := \Big\{\rho>0 : 0 < \H^{n-1}\big(E\cap S(\rho)\big)<\H^{n-1}\big(S(\rho\big)\Big\}
\]
is an interval, then equality holds in~(\ref{Steiner}) if and only if $E=E^*$ up to rotation about the origin.
\end{theorem}
\begin{proof}
First of all, notice that the volume of any set $E$ can be expressed as
\[
\big| E \big| = \int_{r=0}^{+\infty} f(r) A_E(r)\, dr\,,
\]
so $|E^*|=|E|$ simply because by definition $A_{E^*}\equiv A_E$.\par
Concerning the inequality for the perimeter, recall that spherical caps are the (unique) isoperimetric sets on spheres. Hence, the inequality $P(E^*)\leq P(E)$ follows by integrating in the radial variable and by using Jensen's inequality, exactly as in~\cite[Lemma~3.3]{FMP}.\par
Let us finally consider the case in which the equality $P(E^*)=P(E)$ holds: by the equality case of the isoperimetric inequality on the sphere, we deduce that each section $E\cap S(r)$ must be a spherical cap centered at some $r\theta(r)$ for $\theta(r)\in \S^{n-1}$. Moreover, by the equality case in Jensen's inequality, one derives also that the radial component $\nu(x)\cdot x/|x|$ of the normal vector $\nu(x)$ is constant on every section $E \cap S(r)$. Together with~(\ref{rlp}) and the hypothesis that $I_E$ is an interval, this implies that the function $r \mapsto\theta(r)$ is constant, so $E = E^*$ up to a  rotation about the origin.
\end{proof}
\begin{remark}{\rm
Observe that something more precise than the claim of Theorem~\ref{thmSteiner} can be said. In fact, as the proof shows, if $E$ is an isoperimetric set then the function $r \mapsto \theta(r)$ which associates to any $\rho$ the ``center'' of the spherical slice $E\cap S(r)$ whenever it is not trivial, is locally constant. Thus, in particular, we derive that an isoperimetric set $E$ without tangential boundary --that is, for which (\ref{rlp}) holds-- must coincide with its spherical symmetrization up to a rotation if both $E$ and $\partial E$ are connected. More generally, we can say that each connected component $E_i$ of $E$ for which also $\partial E_i$ is connected must coincide up to a rotation with the corresponding component of $E^*$ if~(\ref{rlp}) holds. Indeed, if $E_i$ and $E_j$ are two distinct connected components of $E$, then $E_i^*\cap E_j^*= \emptyset$, since otherwise there are radial slices which are not spheres, against the isoperimetric property of $E$.}
\end{remark}

\begin{corollary}
Assume that the density $f$ is smooth and radial, and let $E$ be an isoperimetric set with connected boundary and being equal to the closure of its interior. Then $E=E^*$ up to a rotation about the origin.
\end{corollary}
\begin{proof}
By the classical regularity theorems (see Theorem~\ref{thm1.0}), $\partial E$ is a smooth hypersurface except for a singular set of dimension at most $n-8$. Since $\partial E$ is connected, then so are $E$ and $I_E$. Thus, if~(\ref{rlp}) holds true, then the thesis directly follows from Theorem~\ref{thmSteiner}. It remains to consider the case when~(\ref{rlp}) does not hold, thus $\partial E$ is tangential on a set of positive area. At a smooth point of density of such points, $\partial E$ is tangential and $E$ has the same extrinsic curvature and hence the same generalized curvature as the sphere $S$ about the origin. Since both $\partial E$ and $S$ have the same constant generalized curvature, by uniqueness of solutions to elliptic partial differential equations we obtain $\partial E= S$, hence $E$ is a ball and then $E = E^*$ as desired.
\end{proof}

\bigskip

Next we show under suitable assumptions that isoperimetric sets are mean-convex.

\begin{theorem}[{\bf Mean-convexity}\hspace{0pt}]
Let $f$ be a smooth, radial, log-convex density on $\R^n$, $n\geq 2$. Then every bounded connected isoperimetric set is mean-convex at every regular point (thus, convex if the dimension is $n=2$).
\end{theorem}
\begin{proof}
Let $E$ be an isoperimetric set, and let $z$ be the point of $\partial E$ of greatest distance from the origin. By Theorem~\ref{thm1.0}, H is a smooth at z (since the oriented tangent cone lies in a halfspace, it must be a hyperplane). In particular, the classical mean curvature $H_0$ satisfies
\begin{equation}\label{convex}
H_0(z) \geq 0\,.
\end{equation}
Writing $f(x)=e^{v(|x|)}$, for any regular point $x\in \partial E$ we have
\[
\frac{\partial v}{\partial \nu}(z) = v' \big(|z|\big) \geq v'\big(|x|\big) \geq \frac{\partial v}{\partial \nu}(x)\,.
\]
Recalling that the curvature $H_f= H_0 + \partial v / \partial \nu$ is constant on $\partial E$ because $E$ is isoperimetric, as pointed out in Section~\ref{succor}, we get
\[
H_0(x) \geq H_0(z)\,,
\]
so from~(\ref{convex}) we get that $E$ is mean-convex, as required. In particular, if $n=2$ (where every boundary point is regular), then $E$ is convex.
\end{proof}

Similar mean-convexity holds if the isoperimetric profile $\IF$ is non-decreasing, as shown by the following well-known result.

\begin{theorem}
Let $E$ be an isoperimetric set of volume $|E|=V$ and finite mean curvature in $\R^n$ with density $f$. Then the upper right derivative $\IF'_+$ and the lower left derivative $\IF'_-$ satisfy
\begin{equation}\label{formulaLambda'}
\IF'_+(V) \leq H(E) \leq \IF'_-(V)
\end{equation}
(in particular, if $\IF'(V)$ exists, then $H(E) = \IF'(V)$). As a consequence, whenever $\IF$ is nondecreasing the set $E$ is mean-convex at every regular point.
\end{theorem}
\begin{proof}
This is an easy consequence of the first variation formulae~(\ref{expansion-volume}) and~(\ref{expansion-perimeter}). Indeed, choose any continuous function $u:\partial E\to \R$ with compact support such that
\[
\int_{\partial E} u(x) f(x) \, d\H^{n-1}(x) = 1\,.
\]
The $\eps$ expansions $E_\eps$ ($|\eps|\ll 1$) as in Lemma~\ref{lemma-expansion} satisfy
\begin{align*}
\big|E_\eps\big| = |E| + \eps + o(\eps)\,, && P\big( E_\eps\big) = P(E) + \eps H(E) + o(\eps)\,.
\end{align*}
As a consequence, since clearly $\IF(V+\eps)\leq P(E_\eps)$, we can immediately deduce
\[
\IF'_+(V) = \limsup\limits_{\eps\searrow 0}  \frac{\IF(V+\eps) -\IF(V)}{\eps} \leq H(E)\,,
\]
as well as
\[
\IF'_-(V) = \liminf\limits_{\eps\nearrow 0}  \frac{\IF(V+\eps)-\IF(V)}{\eps} \geq H(E)\,.
\]
This establishes~(\ref{formulaLambda'}), and in particular the fact that $H(E) = \IF'(V)$, provided the latter exists. Finally, if $\IF$ is increasing, then $\IF'_+(V)\geq 0$, thus $H(E)\geq 0$.
\end{proof}

\begin{remark}
{\rm Notice that~(\ref{formulaLambda'}) in particular says that $\IF_+'(V)<+\infty$ whenever there is an isoperimetric set of volume $V$ of finite mean curvature. On the other hand, $\IF_+'(V)=+\infty$ in some situations where there is no isoperimetric set of volume $V$. Simple examples of this fact are $\R$ with density $1$ on the unit interval and $2$ outside, or Example~\ref{ex11} in dimension $2$. A more involved example is the one that we gave in Proposition~\ref{steepest}, where $\IF_+'(1)=\infty$ and there are no isoperimetric sets of volume $1$.}
\end{remark}

\bigskip

Let us now consider another question, which is in fact a generalization of convexity to $\R^n$ with density: that a set contain all shortest paths between pairs of points. Throughout the rest of this section, we will work only in $\R^2$ and weight distance by the density. It is important to notice that, in general, both the existence and the uniqueness of shortest paths between two given points may fail. For instance, in Gauss space the existence of shortest paths is not true for all pairs of points, since some shortest paths pass through infinity. Our main
results are Lemma~\ref{lemma5.9} and Proposition~\ref{lemma5.8}, which regard the shape of shortest paths between pairs of points and the question of whether a shortest path between two points inside an isoperimetric set entirely lies inside the set. For the sake of clarity, we start with the following result, which is a particular case of Proposition~\ref{lemma5.8} but whose proof elucidates the idea for the more general result.
\begin{lemma}\label{mild-convexity}
Let $f$ be a smooth or Lipschitz radial density on $\R^2$ such that the isoperimetric profile $\IF$ is non-decreasing, and let $E$ be an open connected isoperimetric set. Moreover, suppose that $x,\, y \in \overline E$ and $y=\lambda x$ with $\lambda>0$. Then the segment $xy$ is contained in $\overline E$.
\end{lemma}
\begin{figure}[htbp]
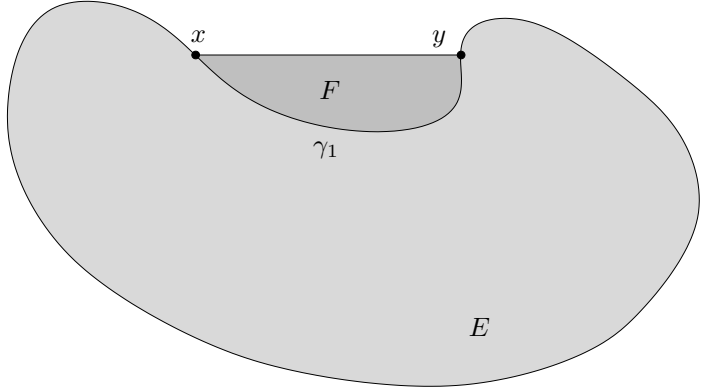\caption{The situation in Lemma~\ref{mild-convexity}.}\label{Figure:G}
\end{figure}
\begin{proof}
Assume that the result is not true. Then, we can take two points $x,\, y\in\partial E$ with $y=\lambda x$, $\lambda>0$, and with the open segment $xy$ entirely contained in $\R^2\setminus \overline E$. For simplicity, assume that $x=\big(|x|,0\big)$ and $y=\big(|y|,0\big)$. Let us then consider the curve $\partial E$: it is the union of two curves, $\gamma_1$ and $\gamma_2$, having $x$ and $y$ as endpoints. Let $F$ be the open bounded subset of $\R^2$ whose boundary is $\gamma_1\cup xy$, where $\gamma_1$ is chosen in such a way that $F\cap E = \emptyset$ as in Figure~\ref{Figure:G}. Finally, let $E^+ = E \cup F$, so that
\[
\partial E^+ = \partial E \cup xy \setminus \gamma_1\,.
\]
Let $\alpha: \gamma_1 \to \R^2$ be the function $\alpha(z) = \big( |z|, 0\big)$. Notice that by radiality $f(z) = f(\alpha(z))$ for all $z\in \gamma_1$, and moreover $\alpha$ is $1-$Lipschitz and $\alpha(\gamma_1) \supseteq xy$. By the very same argument as in Theorem~\ref{Lambda-increasing}, we get that
\[
\int_{\gamma_1} f(z)\, d\H^1(z) > \int_{xy} f(z)\, d\H^1(z)\,,
\]
from which it immediately follows that
\begin{align*}
P(F) < P(E)\,, && |F| > |E|\,.
\end{align*}
This is a contradiction of the assumption that $\IF$ is increasing, since we would get
\[
\IF\big(|F|\big) \leq P(F) < P(E) = \IF\big(|E|\big)\,.
\]
\end{proof}

The following consequence is trivial but interesting.

\begin{corollary}
With the same hypotheses as in Lemma~\ref{mild-convexity}, if moreover $O\subseteq \overline{E}$, then $\overline{E}$ is star-shaped in $O$. In other words, $\partial E$ is a polar graph.
\end{corollary}

We can now give the more general result.

\begin{prop}\label{lemma5.8}
Let $f$ be a smooth or Lipschitz density on $\R^2$ such that the isoperimetric function $\IF$ is nondecreasing. Let $E$ be a connected open isoperimetric set, and $\gamma$ be a shortest path connecting two points $x,\, y \in \overline{E}$. Then one has $\gamma\subseteq\overline{E}$ if\\
{\bf (A)} $\IF$ is \emph{strictly} increasing, or \\
{\bf (B)} $\gamma$ is the \emph{unique} shortest path connecting $x$ and $y$.
\end{prop}
\begin{proof}
If the assertion were not true, we could take two points $x,\, y \in \partial E$ and a shortest path $\gamma$ between $x$ and $y$, with the property that $\gamma$ lies entirely outside $E$. Then, as in the proof of Lemma~\ref{mild-convexity}, notice that $\partial E$ is a closed curve, and it is the union of two curves connecting $x$ and $y$. Let $\tilde\gamma$ be one of these two curves such  that $\gamma \cup \tilde\gamma$ is the boundary of a set $F$ with $E\cap F=\emptyset$. Letting $E^+=E\cup F$, one has that $|E^+| > |E|$, while
\[
\IF\big( |E^+|\big) \leq P(E^+) = P(E) + \ell(\gamma) - \ell(\tilde\gamma)\leq P(E) = \IF\big(|E|\big)\,,
\]
where
\[
\ell(\tau) := \int_0^1 f\big(\tau(s)\big) \big| \tau'(s)\big| \, dx
\]
denotes the length of any Lipschitz curve $\tau: (0,1)\to \R^2$\,. The contradiction is then found either because $\IF$ is strictly increasing (in case {\bf (A)}), or because $\ell (\tilde\gamma) > \ell(\gamma)$ (in case {\bf (B)}).
\end{proof}

We conclude with a simple general result about the shape of shortest paths with a radial non-decreasing density (notice that if the density is radial and non-decreasing then there always exists at least one shortest path between any two given points).

\begin{lemma}\label{lemma5.9}
Let $f$ be a smooth or Lipschitz radial non-decreasing density on $\R^2$, let $P\neq Q$ be two points in $\R^2$, and let $\gamma$ be a shortest path between $P$ and $Q$. Then $\gamma$ lies entirely in the closed Euclidean triangle $\overline{OPQ}$.
\end{lemma}
\begin{proof}
Let $H\subseteq \R^2$ be an open half-space which does not contain the origin $O$, and whose boundary contains both $P$ and $Q$. Of course, there are two half-spaces whose boundary contains $P$ and $Q$, and exactly one of them does not contain the origin, unless the line $PQ$ contains $O$, in which case both the half-spaces have this property. We divide the proof into three steps. 
\step{I}{$\gamma$ does not intersect $H$.}
Let $\alpha:\R^2\to \R^2\setminus H$ be the function given by $\alpha(x)=x$ if $x\in \R^2\setminus H$, while otherwise $\alpha(x)$ is the orthogonal projection of $x$ on $\partial H$. The function $\alpha$ is $1-$Lipschitz, and since $f$ is radial and increasing one has $f\big(\alpha(x)\big)\leq f(x)$ for all $x\in \R^2$. The path $\alpha\circ\gamma$ between $P$ and $Q$ then satisfies
\begin{equation}\label{circ<}
\ell\big(\alpha\circ \gamma \big) = \int f\big(\alpha\circ\gamma(t)\big) \big|(\alpha\circ\gamma)'(t)\big|\, dt
\leq \int f\big(\gamma(t)\big) \big| \gamma'(t)\big| = \ell\big(\gamma\big)\,.
\end{equation}
This implies that $\alpha\circ\gamma$  also is a shortest path, and moreover a quick look at the equality cases above ensures that the inequality~(\ref{circ<}) is strict unless $\alpha\circ\gamma=\gamma$, which means that $\gamma$ cannot intersect the half-space $H$. The first step is complete. In particular, if $P$ is a positive multiple of $Q$, that is, the line $PQ$ contains $O$, then the segment $PQ$ is the \emph{unique} shortest path connecting $P$ and $Q$.

\step{II}{$\gamma$ does not cross the segments $OP$ and $OQ$.}
We want now to show that $\gamma$ cannot cross the segment $OP$ (the same argument, of course, will apply also to $OQ$). This means that, if the image of $\gamma$ contains two points $P_1=\gamma(t_1)$ and $P_2=\gamma(t_2)$ in the segment $OP$, then $\gamma\big([t_1,t_2]\big)$ is the segment $P_1P_2$. This is an immediate consequence of Step~I, since the restriction of $\gamma$ to $[t_1,t_2]$ is a shortest path connecting $P_1$ and $P_2$, one of which is a multiple of the other by definition.

\step{III}{Conclusion.}
We are now ready to conclude the proof. We may assume that neither $P$ nor $Q$ is a multiple of the other, since otherwise the lemma already has been established. By Steps~I and~II, $\gamma$ cannot intersect the half-space $H$, nor cross the segments $OP$ and $OQ$. There are then only two possibilities, namely, either $\gamma$ is entirely contained in the closed triangle $\overline{OPQ}$, or $\gamma$ has an empty intersection with the open triangle $OPQ$. In the first case, we have the result, so let us assume that the second case holds. If $O$ belongs to the path $\gamma$ then we are done, because by Step~I we deduce that $\gamma$ is the union of the segments $PO$ and $OQ$, which violates regularity. On the other hand, if $O$ does not belong to the path $\gamma$, then by construction there must be a point $P' = \gamma(t)$ on $\gamma$ such that $P' = \lambda P$ for some $\lambda <0$. But this would imply, again by Step~I, that the first part of $\gamma$ is a segment between $P$ and $P'$, which is a contradiction of the assumption that $O$ does not belong to the image of $\gamma$. The proof is complete.
\end{proof}

\section{The main existence results\label{seccon}}

In this last section we study existence for increasing densities on $\R^n$. As we discussed in Section~\ref{sect2}, one may expect that existence holds true for non-decreasing densities, but Proposition~\ref{ex:nonex} showed that this is not always the case. Existence probably holds, however, for radial increasing densities.

\begin{conj}\label{conj}
Let $f$ be a radial, increasing density on $\R^n$. Then isoperimetric sets exist for all volumes.
\end{conj}

Other open problems are collected in Section~\ref{sectopenpbs}. We will prove this conjecture under a technical condition (Theorem~\ref{generalcase}) implied by various conditions on the growth of the density (Theorems~\ref{Thm2A}, \ref{6.11}, and~\ref{Thm2B}). Note that by Theorem~\ref{thm2.2}, Conjecture~\ref{conj} holds if the density goes to infinity.

\begin{lemma}\label{6.2}
Consider $\R^n$ endowed with a density approaching a limit $a>0$ at infinity. For any $V > 0$, the isoperimetric profile satisfies
\begin{equation}\label{mainineq}
\IF(V) \leq n \big(\omega_n a\big)^{\frac 1n} \, V^{\frac{n-1}{n}}\,.
\end{equation}
Moreover, one has
\begin{equation}\label{mainineq2}
\frac{\IF(V)}{V^{\frac {n-1}n }} \xrightarrow[\ \ V\to\infty\ \ ]{} n \big(\omega_n a\big)^{\frac 1n}\,.
\end{equation}
\end{lemma}
To prove this lemma, the following definition is quite useful.
\begin{definition}\label{defmd}
Given a ball $B$ in $\R^n$ with density, we will say that its \emph{mean density} is the number $\rho$ such that
\begin{equation}\label{defmeandens}
P(B) = n \big( \omega_n \rho\big)^{\frac 1 n}\, \big|B\big|^{\frac{n-1}{n}}\,.
\end{equation}
\end{definition}
The reason for the choice of the name is very easy to understand: if a ball $B$ of radius $r$ lies in a region where the density is constantly $\rho$, then its volume and perimeter are
\begin{align*}
|B| = \rho \omega_n r^n \,, && P(B) = \rho n\omega_n r^{n-1}\,,
\end{align*}
so by~(\ref{defmeandens}) we get that the mean density of $B$ is $\rho$. If we call $\rho_{\rm min}$ and $\rho_{\rm sup}$ the minimum and the supremum of $f$ inside $B$, it is not true in general that the mean density satisfies $\rho_{\rm min}\leq \rho\leq \rho_{\rm sup}$. What is true is that
\begin{equation}\label{estdens}
\frac{\rho_{\rm min}^{n}}{\rho_{\rm sup}^{n-1}} \leq \rho \leq \frac{\rho_{\rm sup}^{n}}{\rho_{\rm min}^{n-1}}\,;
\end{equation}
both the inequalities are sharp, the extreme cases being when $\rho\equiv \rho_{\rm sup}$ inside $B$ and $\rho\equiv \rho_{\rm min}$ on $\partial B$, or when $\rho\equiv \rho_{\rm min}$ inside $B$ and $\rho\equiv \rho_{\rm sup}$ on $\partial B$ (this latter is only a limiting case since the density would not be lower semicontinuous). To prove~(\ref{estdens}), just note that
\begin{align*}
\rho_{\rm min} \omega_n r^n \leq |B| \leq  \rho_{\rm sup} \omega_n r^n\,, &&
\rho_{\rm min} n\omega_n r^{n-1} \leq P(B) \leq \rho_{\rm sup} n\omega_n r^{n-1}\,.
\end{align*}
We can now give the proof of Lemma~\ref{6.2}.
\begin{proof}[Proof of Lemma~\ref{6.2}]
Fix a volume $V>0$, fix $\eps>0$, and let $B$ be a ball with volume $|B|=V$. If we call, as before, $\rho_{\rm min}$ and $\rho_{\rm sup}$ the minimum and the maximum of $f$ inside $B$, then we can assume that
\begin{align*}
\rho_{\rm min} \geq a - \eps\,,&& \rho_{\rm sup} \leq a+\eps\,,
\end{align*}
provided that the ball is taken sufficiently far from the origin. Hence, by~(\ref{estdens}) we have that the mean density $\rho$ of $B$ satisfies
\[
\rho \leq \frac{(a+\eps)^n}{\big(a-\eps\big)^{n-1}} \leq a+ 3n\eps \,,
\]
provided that $\eps$ is small enough, and thus
\[
\IF(V) \leq P(B) = n \big( \omega_n \rho\big)^{\frac 1 n}\, |B|^{\frac{n-1}{n}}
\leq n \big( \omega_n (a+3n\eps)\big)^{\frac 1 n}\, V^{\frac{n-1}{n}}\,.
\]
Since $\eps>0$ was arbitrary, we get~(\ref{mainineq}).\par
Let us now concentrate on~(\ref{mainineq2}). To show its validity, fix an arbitrary $\eps>0$ and take a big radius $R=R(\eps)$ such that $a-\eps\leq f\leq a+\eps$ out of the ball $B(R)$. Consider now a set $F$ of huge volume $V=|F|$. Then, the Euclidean volume $\VE{F}$ is very big, thus the standard isoperimetric inequality ensures us that
\[
\PE\big(F\setminus B(R)\big) \geq
n \omega_n^{\frac 1n} \Big(\VE{F\setminus B(R)} \Big)^{\frac{n-1}n}
\geq n \omega_n^{\frac 1n} \Big(\VE{F}-\omega_nR^n\Big)^{\frac{n-1}n}\,.
\]
By the definition of $R$ we deduce
\[\begin{split}
P(F) &= \int_{\partial F} f(x) \, d\H^{n-1}(x)
\geq (a-\eps) \,\H^{n-1}\big(\partial F\setminus \overline{B(R)}\big)\\
&\geq (a-\eps)\, \Big(\PE\big(F\setminus B(R)\big) - n\omega_n R^n\Big)
\geq (a-\eps)\, \Big( n \omega_n^{\frac 1n} \Big(\VE{F}-\omega_nR^n\Big)^{\frac{n-1}n}- n\omega_n R^n\Big)\\
&\geq(a-\eps)\,\Bigg(n\omega_n^{\frac 1n}\bigg(\frac{V-\big|B(R)\big|}{a+\eps}-\omega_nR^n\bigg)^{\frac{n-1}n}-n\omega_n R^n\Bigg)
\geq (a-3\eps)^{\frac 1n} n \omega_n^{\frac 1n}\, V^{\frac{n-1} n}\,,
\end{split}\]
where the last inequality holds provided $V \gg R$. Since $F$ was an arbitrary set of volume $V$, we deduce that
\[
\IF(V) \geq (a-3\eps)^{\frac 1n} n \omega_n^{\frac 1n}\, V^{\frac{n-1} n}
\]
and finally, since $\eps$ was arbitrary, also recalling~(\ref{mainineq}) we get~(\ref{mainineq2}).
\end{proof}

We can now show that Conjecture~\ref{conj} holds true if a particular technical condition holds. Recall that, thanks to Theorem~\ref{thm2.2}, if a density $f$ is radial then we may assume that it approaches a finite limit at infinity.

\begin{prop}\label{generalcase}
Let $f$ be a density on $\R^n$ approaching a finite limit $a>0$ at infinity, and assume that isoperimetric sets are bounded (see Corollary~\ref{corbdd}). Let $V>0$ and suppose that, for any $0<\widetilde V\leq V$, there exists a ball of volume $\widetilde V$ arbitrarily far from the origin and having mean density at most $a$. Then there exists an isoperimetric set of volume $V$.
\end{prop}

In the proof of the above proposition, we will need the following simple useful result, which holds without any assumptions on the density $f$.
\begin{lemma}\label{folllem}
Let $f$ be any density on $\R^n$ and $V>0$. Let $F_i$ be a sequence of sets of volume $V$ such that $P(F_i)\to \IF(V)$, and such that $\Chi{F_i}\xrightarrow{\,L^1_{\rm loc}\,}\Chi{F}$ for some set $F$. Then, $F$ is an isoperimetric set of volume $|F|$. In addition, if $f$ approaches a limit $a>0$ at infinity, then the following estimate holds,
\begin{equation}\label{follest}
\IF(V)\geq  P(F) + n \big(\omega_n a\big)^{\frac 1 n} \big( V -|F|\big)^{\frac{n-1} n}\,.
\end{equation}
\end{lemma}
\begin{proof}
Take some positive number $r>0$, let $F^+=B(F,r)$ be the $r-$neighborhood of $F$, and for any $i\in\N$ define
\begin{align*}
G_i = F_i \cap F^+\,, && H_i = F_i \setminus F^+\,.
\end{align*}
Since $\big| F_i \setminus F\big|\to 0$ for $i\to \infty$, for all $r>0$ except at most countably many one has
\begin{equation}\label{aer}
\H^{n-1}_f \big( F_i \cap \partial F^+ \big) \to 0\,.
\end{equation}
Observe now that
\begin{equation}\label{aer0}
P\big(F_i\big) = P\big(G_i\big) + P\big( H_i\big) -2 \H^{n-1}_f \big( F_i \cap \partial F^+\big)\,.
\end{equation}
Since $\Chi{G_i}\xrightarrow{\,L^1_{\rm loc}\,}\Chi{F}$ we have by lower semicontinuity
\begin{equation}\label{aer1}
P(F)  \leq \liminf P(G_i)\,.
\end{equation}
Assume that $F$ is not an isoperimetric set of volume $|F|$. Then there exists a bounded set $E$ with $|E|>|F|$ and $P(E)<P(F)$. Take now $r>0$ satisfying~(\ref{aer}) and big enough so that $E\subset\subset F^+$, and let
\[
\widetilde F_i = \big(F_i \setminus F^+ \big)\cup E= H_i \cup E \,.
\]
By construction, $\big| \widetilde F_i\big| > V$ for all $i\gg 1$. Using the fact that $E\subset\subset F^+$, Theorem~\ref{Lambda-increasing}, (\ref{aer0}) and~(\ref{aer1}), one then finds
\[\begin{split}
\IF(V) &\leq \liminf P\big(\widetilde F_i\big) = \liminf \Big(P\big(H_i)+ P(E)\Big)\\
&= \liminf \Big(P\big(F_i\big) - P\big(G_i\big) +2 \H^{n-1}_f \big( F_i \cap \partial F^+\big)\Big)+ P(E)
= \IF(V)-P(F)+P(E)\\
&<\IF(V)\,.
\end{split}\]
This contradiction shows that $F$ is isoperimetric.\par
Let us show now the second part of the claim. To do so, assume that the density $f$ approaches a limit $a$ at infinity, and fix some $\eps>0$. If we take $r$ big enough, depending on $\eps$, then it is admissible to assume that $f\geq a-\eps$ outside $B(F,r)$. Arguing exactly as in the proof of~(\ref{estdens}) we can then immediately deduce that
\begin{equation}\label{aer2}
\liminf P(H_i) \geq n \big( \omega_n (a-\eps)\big)^{\frac 1n} \big( V - |F|\big)^{\frac{n-1}n}\,.
\end{equation}
Inserting~(\ref{aer1}) and~(\ref{aer2}) into~(\ref{aer0}), recalling that $P(F_i)\to \IF(V)$, and assuming without loss of generality that~(\ref{aer}) holds for $r$, we immediately get that
\[
\IF(V)\geq  P(F) + n \big(\omega_n (a-\eps)\big)^{\frac 1 n} \big( V -|F|\big)^{\frac{n-1} n}\,.
\]
Since $\eps>0$ was arbitrary, the validity of~(\ref{follest}) follows.
\end{proof}
We now obtain Proposition~\ref{generalcase} as an easy corollary.
\begin{proof}[Proof of Proposition~\ref{generalcase}]
Let us start by taking a sequence of sets $F_i$ such that $|F_i|=V$ and $P(F_i)\rightarrow \IF(V)$. Up to subsequences, there exists a set $F$ such that $\Chi{F_i}\xrightarrow{\,L^1_{\rm loc}\,}\Chi{F}$, hence we can apply Lemma~\ref{folllem}. If $|F|=V$, then we are done, since $F$ is an isoperimetric set of volume $V$. Otherwise, assume that $\widetilde V= V- |F| >0$, and recall that $F$ is bounded by assumption. Therefore, we can take a ball $B$ of volume $\widetilde V$ having mean density smaller than $a$ and not intersecting $F$. Hence by~(\ref{follest})
\[
P\big( B \cup F \big) = P(B) + P(F)
\leq n \big( \omega_n a \big)^{\frac 1 n}\, |B|^{\frac{n-1}{n}} + P(F)
\leq \IF(V)\,.
\]
Thus $B\cup F$ is an isoperimetric set of volume $V$, and the proof is complete.
\end{proof}

A similar argument shows that Conjecture~\ref{conj} is true if the volume $V$ is sufficiently small.

\begin{prop}\label{Prop6.6}
Let $f$ be a continuous density on $\R^n$ approaching a finite limit $a>0$ at infinity, and such that $f(\hat x)<a$ for some $\hat x\in\R^n$. Then there exist isoperimetric sets for all small volumes. In other words, there is some $V_0>0$ such that there exists an isoperimetric set of volume $V$ for each $0<V< V_0$.
\end{prop}
\begin{proof}
Let us start arguing exactly as in Proposition~\ref{generalcase}: let $F$ be the $L^1_{\rm loc}$ limit of a sequence of sets of volume $V$ minimizing the perimeter, which is an isoperimetric set by Lemma~\ref{folllem}. Let $\widetilde V= V - |F|$, and assume that $\widetilde V>0$ since otherwise $F$ is already the required isoperimetric set. By the assumption $f(\hat x)<a$ and by the continuity of $f$, there exist some $\delta>0$ and $r>0$ such that every ball $B$ contained in $B(\hat x,r)$ has mean density smaller than $a-\delta$. Let then $\eta>0$ be a small constant to be determined later: if $V$ is small compared to $\big|B(\hat x,r)\big|$, there exists some ball $B\subseteq B(\hat x,r)$ such that
\begin{align*}
\big| F \cup B\big| = V\,, && \big| F \cap B\big|\leq \eta \,.
\end{align*}
Recalling~(\ref{follest}), if $\eta$ is small enough compared to $\delta$ we can estimate
\[\begin{split}
\IF(V) &\leq P\big( F \cup B\big) \leq P(F) + P(B)\leq
P(F) + n\big(\omega_n(a-\delta)\big)^{\frac 1n} |B|^{\frac{n-1}n}\\
&\leq P(F) + n\big(\omega_n(a-\delta)\big)^{\frac 1n} \big|\widetilde V +\eta\big|^{\frac{n-1}n}
< P(F) + n\big(\omega_na\big)^{\frac 1n} \widetilde V ^{\frac{n-1}n}<\IF(V)\,,
\end{split}\]
which gives a contradiction. Therefore, we deduce that necessarily $\widetilde V=0$, so that the existence of an isoperimetric set is given by $F$ itself.
\end{proof}

\begin{remark}{\rm
Notice that in the above proof we showed something stronger than the existence of isoperimetric sets for small volume $V$. In fact, we have proved that for every minimizing sequence there is no mass vanishing at infinity.
}\end{remark}

When the density is smooth, we can now strengthen Proposition~\ref{Prop6.6} to conclude that small isoperimetric sets are ${\rm C}^1$ close to round balls.

\begin{prop}
Let $f$ be a smooth density on $\R^n$ approaching a finite limit $a > \inf_{\R^n} f$ at infinity. Then for small volume isoperimetric sets exist and are smoothly close to round balls near a point of minimum density.
\end{prop}
\begin{proof}
Let $F$ be an isoperimetric set of sufficiently small volume, which exists by Proposition~\ref{Prop6.6}. The proof that $F$ is ${\rm C}^\infty$ close to a round ball follows the argument in~\cite[Sect. 2]{MJ}, which we now summarize and which may be consulted for details and further references. By Heinze-Karcher, the classical mean curvature of $\partial F$ is bounded by $C_1 P(F)/|F|$, which in turn is less than $C_2 |F|^{-\frac 1 n}$. By  monotonicity, the surface area inside a ball of radius $|F|^{\frac 1n}$ about a point of the surface is at least $C_3 |F|^{\frac{n-1} n}$. Since $P(F) \leq C_4 |F|^{\frac{n-1}n}$, we get that $\partial F$ and hence $F$ is contained inside $C_5$ balls of radius $C_6 |F|^{\frac 1n}$. Let $|F|$ approach $0$ and scale up each ball by $|F|^{-\frac 1n}$. The limit is an isoperimetric region in $C_5$ copies of $\R^n$ each with constant density, hence itself is made by $C_5$ sets which are either balls or empty sets. Since the limit is an isoperimetric set, it must be a single ball and $C_5-1$ empty sets. Since mean curvature is bounded, ${\rm C}^{1,\alpha}$ convergence follows by Allard's regularity theorem. Higher order convergence follows by Schauder estimates (also see \cite[Prop. 3.3]{M2}). An easy computation shows that it is best for the set to be near a point of minimum density.
\end{proof}

We now give our first major existence theorem, under a hypothesis that the density approaches the limiting value slowly.

\begin{theorem}\label{Thm2A}
Let $f$ be a density on $\R^n$ approaching a finite limit $a>0$ at infinity, and assume that the isoperimetric sets are bounded. Suppose that, for every $\widetilde V> 0$ and for every $R_0>0$, there is some ball $B$ of volume $\widetilde V$ at distance from the origin at least $R_0$ such that
\begin{equation}\label{ass2Afin}
\sup_{x\in B} f(x) \leq a^{\frac 1n} \Big( \inf_{x\in B} f(x)\Big)^{\frac{n-1} n}\,.
\end{equation}
Then there exist isoperimetric sets of all volumes.
\end{theorem}
\begin{proof}
Fix a volume $V$, and let
\[
\tau = 3\,\bigg( \frac{V}{\omega_n a}\bigg)^{\frac 1n}\,.
\]
Since the density approaches $a$ at infinity, for some $R_{\rm min}$ any ball of volume $V$ at distance at least $R_{\rm min}$ from the origin has diameter less than $\tau$. Let us then fix $0 < \widetilde V \leq V$, and let $R_0\geq R_{\rm min}$ be arbitrarily big. Thanks to Proposition~\ref{generalcase}, it suffices to find a ball of volume $\widetilde V$ at distance bigger than $R_0$ from the origin with mean density at most $a$.\par
By assumption, there exists a ball $B$ of volume $\widetilde V$ such that~(\ref{ass2Afin}) holds. As explained above, we conclude by checking that the mean density $\rho$ of $B$ is at most $a$, which follows since by~(\ref{ass2Afin}) one has
\[
\rho\leq\frac{\rho_{\rm sup}^{n}}{\rho_{\rm min}^{n-1}}=\frac{\Big(\sup_{x\in B} f(x)\Big)^n}{\Big( \inf_{x\in B} f(x)\Big)^{n-1}}\leq a\,.
\]
\end{proof}
\begin{remark}\label{6.10}{\rm
Notice that, in the particular case when $f$ is radial and nondecreasing, then~(\ref{ass2Afin}) can be rewritten in the following particularly useful way. For any $R_0>0$ and any $\tau>0$, there exists $R>R_0$ such that 
\begin{equation}\label{ass2A}
f(R+\tau) \leq a^{\frac 1n} f(R)^{\frac{n-1} n}\,.
\end{equation}}
\end{remark}
\begin{proof}
Assume first that~(\ref{ass2Afin}) holds. Fix any $\tau>0$ and any $R_0>0$, and let $\widetilde V= a \omega_n (\tau/2)^n$. By assumption, there exists some ball $B=B(\bar x,r)$ having distance from the origin bigger than $R_0$ and such that~(\ref{ass2Afin}) is true. Since $f$ is nondecreasing and converging to $a$, one has that $r>\tau/2$. Therefore, calling $R=|\bar x|-r>R_0$, one finds
\[
f(R+\tau)\leq f (R+2r) = \sup_{x\in B} f(x) \leq a^{\frac 1n} \Big( \inf_{x\in B} f(x)\Big)^{\frac{n-1} n}
= a^{\frac 1n} f(R)^{\frac{n-1} n}\,,
\]
that is, (\ref{ass2A}).\par
Conversely, assume that~(\ref{ass2A}) is true. Given any $\widetilde V>0$ and any $R_0>0$, we let 
\[
\tau = 3 \bigg(\,\frac{\widetilde V}{a \omega_n}\,\bigg)^{\frac 1n}\,.
\]
By assumption, there exists some $R> R_0$ such that~(\ref{ass2A}) holds. Let then $r>0$ be such that, calling $\bar x$ any point having distance $R+r$ from the origin, and denoting $B=B(\bar x,r)$, one has $|B| = \widetilde V$. Of course there is exactly one such $r$, and since the density is converging to $a$ then one has $r<\tau/2$ provided $R$ is big enough. Hence, recalling again that $f$ is radial and nondecreasing, one gets
\[
\sup_{x\in B} f(x) =f (R+2r) \leq f(R+\tau) \leq a^{\frac 1n} f(R)^{\frac{n-1} n} =a^{\frac 1n} \Big( \inf_{x\in B} f(x)\Big)^{\frac{n-1} n}\,,
\]
that is, (\ref{ass2Afin}).
\end{proof}

We now give our second existence result. Although it follows from Theorem~\ref{Thm2A}, its hypothesis is generally easier to check.

\begin{theorem}\label{6.11}
Let $f$ be a radial, nondecreasing density on $\R^n$ approaching a finite limit $a>0$ at infinity, and assume that, for any $c>0$ and any $\rho>0$, there exists some $R\geq \rho$ such that
\begin{equation}\label{lastone}
f(R) \leq a - e^{-cR}\,.
\end{equation}
Then, there exist isoperimetric regions of every volume.
\end{theorem}
\begin{proof}
We will obtain the result as consequence of Theorem~\ref{Thm2A} (which can be used thanks to Theorem~\ref{thba}). Indeed, suppose that there exists some volume for which no isoperimetric set exists. 

Then by Theorem~\ref{Thm2A} and Remark~\ref{6.10}, there must be some $\tau>0$ and some $R_0>0$ such that for all $R\geq R_0$ one has
\[
f(R+\tau) > a^{\frac 1n} f(R)^{\frac{n-1} n}\,.
\]
Applying this inequality to $R=R_0+\tau$, which is of course bigger than $R_0$, one finds
\[
f(R_0+2\tau) > a^{\frac 1n} f(R_0+\tau)^{\frac{n-1} n} 
> a^{\frac 1n} \bigg(a^{\frac 1n} f(R_0)^{\frac{n-1} n}\bigg)^{\frac{n-1} n} 
=a^{\frac 1n + \frac{n-1}{n^2}}  f(R_0)^{\big(\frac{n-1}n \big)^2}\,,
\]
and an immediate induction argument gives, for any positive integer $k$,
\[
f(R_0+k\tau) > a^{1-\big(\frac{n-1}n \big)^k} f(R_0)^{\big(\frac{n-1}n \big)^k}\,,
\]
which can be rewritten as
\begin{equation}\label{apply}
\frac{f(R_0+k\tau)}{a} > \bigg(\frac{f(R_0)}{a}\bigg)^{\big(\frac{n-1}n \big)^k} > 1 + \bigg(\frac{n-1}n \bigg)^k \ln \frac{f(R_0)}{a}\,.
\end{equation}
Take now any $R\gg R_0$, and call $R'$ the biggest number smaller than $R$ of the form $R'=R_0 + k\tau$ for some $k\in\N$, so that $R-\tau < R' \leq R$. Applying~(\ref{apply}) we find that
\begin{equation}\label{absurd}\begin{split}
f(R) \geq f(R') & > a + a\, \bigg(\frac{n-1}n \bigg)^{\frac{R'-R_0}\tau} \ln \frac{f(R_0)}{a} 
> a + a\, \bigg(\frac{n-1}n \bigg)^{\frac{R-\tau-R_0}\tau} \ln \frac{f(R_0)}{a}\\
&> a- e^{-cR}\,,
\end{split}\end{equation}
where the last equality holds for any
\[
0 < c <  - \, \frac{1}\tau \, \ln \frac{n-1}{n}
\]
provided $R$ is big enough. The validity of~(\ref{absurd}) for every large $R$ contradicts~(\ref{lastone}), completing the proof.
\end{proof}

\begin{corollary}\label{6.7}
Let $f$ be a radial, nondecreasing density on $\R^n$ approaching a finite limit $a>0$ at infinity, and assume that $f$ is of class ${\rm C}^1$ outside some ball $B(R_{\rm min})$. Suppose moreover that
\begin{equation}\label{ass2A'}
f'(R) = o\big(a-f(R)\big)
\end{equation}
for $R\to \infty$. Then there exist isoperimetric sets of all volumes.
\end{corollary}
\begin{proof}
This follows directly by Theorem~\ref{6.11}, just observing that~(\ref{ass2A'}) implies~(\ref{lastone}).
\end{proof}

We finally give our last major existence theorem, under a hypothesis that the average density around the boundary of some balls is not much greater than the average over the balls themselves.
\begin{theorem}\label{Thm2B}
Let $f$ be a density on $\R^n$ approaching a finite limit $a>0$ at infinity, and assume that the isoperimetric sets are bounded. Suppose that, for any $V>0$, there exist balls $B$ of volume $V$ arbitrarily far from the origin satisfying the mean inequality
\begin{equation}\label{xint}
\intmed_{\partial B} f \leq a^{\frac 1n} \bigg( \intmed_B f\bigg)^{\frac{n-1} n}\,.
\end{equation}
Then there exist isoperimetric sets of all volumes.
\end{theorem}
\begin{proof}
This is again an easy consequence of Proposition~\ref{generalcase}. It is enough to show that the mean density of a ball $B$ for which~(\ref{xint}) holds true is less than $a$. Indeed, letting $r$ denote the radius of $B$, one has
\[\begin{split}
P(B) &= n\omega_n r^{n-1} \intmed_{\partial B} f
\leq  n\omega_n r^{n-1} a^{\frac 1n} \bigg( \intmed_B f\bigg)^{\frac{n-1} n}
=  n\omega_n r^{n-1} a^{\frac 1n} \bigg( \frac{|B|}{\omega_n r^n}\bigg)^{\frac{n-1} n}\\
&=  n\big(\omega_n a\big)^{\frac 1n} |B|^{\frac{n-1} n}\,,
\end{split}\]
which by Definition~\ref{defmd} yields that the mean density of $B$ is less than $a$.
\end{proof}
\begin{corollary}\label{Thm2D}
Let $f$ be a density on $\R^n$ approaching a finite limit $a>0$ at infinity, and assume that the isoperimetric sets are bounded. If $f$ is of class ${\rm C}^2$ and superharmonic (at least far from the origin), then there exist isoperimetric sets of all volumes.
\end{corollary}
\begin{proof}
It suffices to observe that, if $B$ is a ball on which $f$ is superharmonic, then
\[
\intmed_{\partial B}  f \leq \intmed_B f \leq a^{\frac 1n} \bigg(\intmed_B f\bigg)^{\frac{n-1} n}\,,
\]
and thus the result follows directly from Theorem~\ref{Thm2B}.
\end{proof}

\begin{remark}{\rm
If $f$ is a ${\rm C}^2$ radial density, then superharmonicity corresponds to
\[
f''(r) \leq - \,\frac{n-1}r\, f'(r)\,,
\]
which is a bit stronger than the concavity $f''(r)\leq 0$. It is an open question whether concavity suffices (see Open Problem~\ref{opconcavity}).
}\end{remark}

\begin{remark}\label{Rem7.16}{\rm
Our results cover all the standard examples. Densities such as $1 -r^{-\alpha}$ ($\alpha > 0$) approach $a = 1$ slowly enough to be covered by Theorem~\ref{6.11}. The density $1 - e^{-r}$ grows too fast for Theorem~\ref{6.11} but it is covered by Corollary~\ref{Thm2D}, which also handles $1 - e^{-e^r}$ and so on for faster growth, although it does not work with slow growth such as $1 - r^{-\alpha}$ if $\alpha < n-2$.\par
Our results may not, however, cover uneven growth. For example, take a density that approaches $1$ rapidly, such as $f(r)=1 - e^{-r}$, so that all of our slow growth theorems do not apply (but this density is covered by Corollary~\ref{Thm2D}). Then make a smooth perturbation which alters only slightly $f$ and $f'$ but makes big changes in $f''$, so that superharmonicity fails in all balls of a given radius and Corollary~\ref{Thm2D} does not apply. It is not clear to us whether or not Theorem~\ref{Thm2B} will apply.
}\end{remark}


\section{Open problems\label{sectopenpbs}}

In this last section, we briefly collect some interesting open problems which are strictly related to the results of this paper.

\begin{op}\label{first}
Is Conjecture~\ref{conj} true? Or, at least, is it true in the particular case $n=2$?
\end{op}

\begin{op}
Is Theorem~\ref{thm2.2} true, for dimension $n=2$, without the assumption that $f$ is radial, or that $f$ diverges? Of course, a positive result would be much stronger than Open Problem~\ref{first} for the case $n=2$. On the other hand, in the proof of Theorem~\ref{thm2.2} for $n = 2$ the divergence assumption played a very minor role.
\end{op}

\begin{op}
Is it true for a radial density plus some reasonable further assumption that the origin must be contained in every isoperimetric region? Recall that the radial assumption is not enough by Example~\ref{ex11}.
\end{op}

\begin{op}
Is it true, in some of the existence results of Section~\ref{seccon}, that isoperimetric regions must be bounded?
\end{op}

\begin{op}\label{opconcavity}
Is it true that isoperimetric regions exist whenever the density is radial and concave (that is, $f(x)=g(|x|)$ for a concave function $g$)?
\end{op}

\section*{Acknowledgments}

This work began when F.M. was visiting Pavia in June 2010, and continued at the Fields Institute in Toronto in October 2010. F.M. also acknowledges partial support from the National Science Foundation. A.P. acknowledges the partial support of the ERC Starting Grant n. 258685 and the ERC Advanced Grant n. 226234. We would like to thank Sean Howe and Emanuel Milman for helpful comments.

\end{document}